\documentclass{amsart}
\usepackage{amsmath,amssymb,url}
\usepackage[latin1]{inputenc}
\usepackage{amsthm}
\usepackage{amsfonts}
\usepackage{enumerate}

\numberwithin{equation}{section}

\theoremstyle{plain}
\newtheorem{theorem}{Theorem}[section]

\newtheorem{lemma}[theorem]{Lemma}
\newtheorem{corollary}[theorem]{Corollary}
\theoremstyle{remark}
\newtheorem{remark}[theorem]{Remark}
\newtheorem{example}[theorem]{Example}

\newcommand{\cl}[1]{\ensuremath{\mathcal{#1}}}
\newcommand{\vect}[1]{\ensuremath{\mathbf{#1}}}
\newcommand{\mtrx}[1]{\ensuremath{\mathbf{#1}}}
\newcommand{\card}[1]{\ensuremath{\lvert{#1}\rvert}}
\DeclareMathOperator{\Pol}{Pol}
\DeclareMathOperator{\Inv}{Inv}
\newcommand{\nat}{\ensuremath{\omega}}
\begin{document}

\title[Galois connection for closed sets of operations]{Galois connection for sets of operations closed under permutation, cylindrification and composition}

\author{Miguel Couceiro}
\address[M. Couceiro]{Mathematics Research Unit \\
University of Luxembourg \\
6, rue Richard Coudenhove-Kalergi \\
L-1359 Luxembourg \\
Luxembourg}
\email{miguel.couceiro@uni.lu}

\author{Erkko Lehtonen}
\address[E. Lehtonen]{Computer Science and Communications Research Unit \\
University of Luxembourg \\
6, rue Richard Coudenhove-Kalergi \\
L-1359 Luxembourg \\
Luxembourg}
\email{erkko.lehtonen@uni.lu}

%\subjclass[2000]{Primary: 08A40; Secondary: 06A15}

%\keywords{function algebra, Galois connection, system of pointed multisets, permutation of variables, cylindrification, composition}

\begin{abstract}
We consider sets of operations on a set $A$ that are closed under permutation of variables, addition of dummy variables and composition. We describe these closed sets in terms of a Galois connection between operations and systems of pointed multisets, and we also describe the closed sets of the dual objects by means of necessary and sufficient closure conditions. Moreover, we show that the corresponding closure systems are uncountable for every $A$ with at least two elements.
\end{abstract}

\maketitle

%%%%%%%%%%%%%%%%%%%%%%%%%%%%%%%%%%%%%%%%%%%%%%%

\section{Preliminaries}
\label{sec:preli}

Throughout this paper, let $A$ be an arbitrary nonempty set. An \emph{operation} on $A$ is a map $f \colon A^n \to A$ for some integer $n \geq 1$, called the \emph{arity} of $f$. For $n \geq 1$, the set of all $n$-ary operations on $A$ is denoted by $\cl{O}_A^{(n)}$, and the set of all operations on $A$ is denoted by $\cl{O}_A := \bigcup_{n \geq 1} \mathcal{O}_A^{(n)}$. For a subset $\cl{F} \subseteq \cl{O}_A$ and an integer $n \geq 1$, the \emph{$n$-ary part} of $\cl{F}$ is defined as $\cl{F}^{(n)} := \cl{F} \cap \cl{O}_A^{(n)}$. For each $n \geq 1$ and $1 \leq i \leq n$, the $n$-ary operation $e_i^{n,A}$ on $A$ defined by $(a_1, \dots, a_n) \mapsto a_i$ is called the $i$-th $n$-ary \emph{projection} on $A$. We denote the set of all projections on $A$ by $\cl{E}_A := \{e_i^{n,A} \mid 1 \leq i \leq n\}$.

We denote the set of natural numbers by $\nat := \{0, 1, 2, \dotsc\}$, and we regard its elements as ordinals, i.e., $n \in \nat$ is the set of lesser ordinals $\{0, 1, \dotsc, n - 1\}$. Thus, an $n$-tuple $\vect{a} \in A^n$ is formally a map $\vect{a} \colon \{0, 1, \dotsc, n - 1\} \to A$. The notation $(a_i \mid i \in n)$ means the $n$-tuple mapping $i$ to $a_i$ for each $i \in n$. The notation $(a_1, \dotsc, a_n)$ means the $n$-tuple mapping $i$ to $a_{i+1}$ for each $i \in n$.

We view an $m \times n$ matrix $\mtrx{M} \in A^{m \times n}$ with entries in $A$ as an $n$-tuple of $m$-tuples $\mtrx{M} := (\vect{a}^1, \dotsc, \vect{a}^n)$. The $m$-tuples $\vect{a}^1, \dotsc, \vect{a}^n$ are called the \emph{columns} of $\mtrx{M}$. For $i \in m$, the $n$-tuple $\bigl( \vect{a}^1(i), \dotsc, \vect{a}^n(i) \bigr)$ is called \emph{row} $i$ of $\mtrx{M}$. If for $1 \leq i \leq p$, $\mtrx{M}_i := (\vect{a}^i_1, \dotsc, \vect{a}^i_{n_i})$ is an $m \times n_i$ matrix, then we denote by $[\mtrx{M}_1 | \mtrx{M}_2 | \dotsb | \mtrx{M}_p]$ the $m \times \sum_{i = 1}^p n_i$ matrix $(\vect{a}^1_1, \dotsc, \vect{a}^1_{n_1}, \vect{a}^2_1, \dotsc, \vect{a}^2_{n_2}, \dotsc, \vect{a}^p_1, \dotsc, \vect{a}^p_{n_p})$. An \emph{empty matrix} has no columns and is denoted by $()$.

For a function $f \colon A^n \to B$ and a matrix $\mtrx{M} := (\vect{a}^1, \dotsc, \vect{a}^n) \in A^{m \times n}$, we denote by $f \mtrx{M}$ the $m$-tuple $\bigl( f (\vect{a}^1(i), \dotsc, \vect{a}^n(i)) \bigm| i \in m \bigr)$ in $B^m$, in other words, $f \mtrx{M}$ is the $m$-tuple obtained by applying $f$ to the rows of $\mtrx{M}$.

Mal'cev~\cite{Malcev} introduced the operations $\zeta$, $\tau$, $\Delta$, $\nabla$, $\ast$ on the set $\cl{O}_A$ of all operations on $A$, defined as follows for arbitrary $f \in \cl{O}_A^{(n)}$, $g \in \cl{O}_A^{(m)}$:
\begin{align*}
(\zeta f)(x_1, x_2, \dotsc, x_n) &:= f(x_2, x_3, \dotsc, x_n, x_1), \\
(\tau f)(x_1, x_2, \dotsc, x_n) &:= f(x_2, x_1, x_3, \dotsc, x_n), \\
(\Delta f)(x_1, x_2, \dotsc, x_{n-1}) &:= f(x_1, x_1, x_2, \dotsc, x_{n-1})
\end{align*}
for $n > 1$, $\zeta f = \tau f = \Delta f := f$ for $n = 1$, and
\begin{align*}
(\nabla f)(x_1, x_2, \dotsc, x_{n+1}) &:= f(x_2, \dotsc, x_{n+1}), \\
(f \ast g)(x_1, x_2, \dotsc, x_{m+n-1}) &:= f \bigl( g(x_1, x_2, \dotsc, x_m), x_{m+1}, \dotsc, x_{m+n-1} \bigr).
\end{align*}
The operations $\zeta$ and $\tau$ are collectively referred to as \emph{permutation of variables,} $\Delta$ is called \emph{identification of variables} (also known as \emph{diagonalization}), $\nabla$ is called \emph{addition of a dummy variable} (or \emph{cylindrification}), and $\ast$ is called \emph{composition.} The algebra $(\cl{O}_A; \zeta, \tau, \Delta, \nabla, \ast)$ of type $(1, 1, 1, 1, 2)$ is called the \emph{full iterative algebra} on $A$, and its subalgebras are called \emph{iterative algebras} on $A$. A subset $\cl{F} \subseteq \cl{O}_A$ is called a \emph{clone} on $A$, if it is the universe of an iterative algebra on $A$ that contains all projections $e_i^{n,A}$, $1 \leq i \leq n$.

A \emph{Galois connection} between sets $A$ and $B$ is a pair $(\sigma, \pi)$ of mappings $\sigma \colon \mathcal{P}(A) \to \mathcal{P}(B)$ and $\pi \colon \mathcal{P}(B) \to \mathcal{P}(A)$ between the power sets $\mathcal{P}(A)$ and $\mathcal{P}(B)$ such that for all $X, X' \subseteq A$ and all $Y, Y' \subseteq B$ the following conditions are satisfied:
\begin{align*}
X \subseteq X' &\Longrightarrow \sigma(X) \supseteq \sigma(X'), \\
Y \subseteq Y' &\Longrightarrow \pi(Y) \supseteq \pi(Y'),
\end{align*}
and
\begin{align*}
X &\subseteq \pi(\sigma(X)), \\
Y &\subseteq \sigma(\pi(Y)),
\end{align*}
or, equivalently,
\[
X \subseteq \pi(Y) \Longleftrightarrow \sigma(X) \supseteq Y.
\]

Galois connections can be equivalently described as certain mappings induced by \emph{polarities,} i.e., relations $R \subseteq A \times B$, as the following well-known theorem shows (for early references, see~\cite{Everett,Ore}; see also~\cite{DEW,Lau}):
\begin{theorem}
\label{thm:Galois}
Let $A$ and $B$ be nonempty sets and let $R \subseteq A \times B$. Define the mappings $\sigma \colon \mathcal{P}(A) \to \mathcal{P}(B)$, $\pi \colon \mathcal{P}(B) \to \mathcal{P}(A)$ by
\begin{align*}
\sigma(X) &:= \{y \in B \mid \forall x \in X \colon (x, y) \in R\}, \\
\pi(Y) &:= \{x \in A \mid \forall y \in Y \colon (x, y) \in R\}.
\end{align*}
Then the pair $(\sigma, \pi)$ is a Galois connection between $A$ and $B$.
\end{theorem}

A prototypical example of a Galois connection is given by the $\Pol$--$\Inv$ theory of functions and relations. For $m \geq 1$, we denote
\[
\cl{R}_A^{(m)} := \{R \mid R \subseteq A^m\} = \mathcal{P}(A^m)
\]
and
\[
\cl{R}_A := \bigcup_{m \geq 1} \cl{R}_A^{(m)}.
\]
Let $R \in \cl{R}_A^{(m)}$. For a matrix $\mtrx{M} \in A^{m \times n}$, we write $\mtrx{M} \prec R$ to mean that the columns of $\mtrx{M}$ are $m$-tuples from the relation $R$. An operation $f \colon A^n \to A$ is said to \emph{preserve} $R$ (or $f$ is a \emph{polymorphism} of $R$, or $R$ is an \emph{invariant} of $f$), denoted $f \vartriangleright R$, if for all $m \times n$ matrices $\mtrx{M} \in A^{m \times n}$
\[
\mtrx{M} \prec R
\quad\text{implies}\quad
f \mtrx{M} \in R.
\]

For a relation $R \in \cl{R}_A$, we denote by $\Pol R$ the set of all operations $f \in \cl{O}_A$ that preserve the relation $R$. For a set $\cl{Q} \subseteq \cl{R}_A$ of relations, we let $\Pol \cl{Q} := \bigcap_{R \in \cl{Q}} \Pol R$. The sets $\Pol R$ and $\Pol \cl{Q}$ are called the sets of all \emph{polymorphisms} of $R$ and $\cl{Q}$, respectively. Similarly, for an operation $f \in \cl{O}_A$, we denote by $\Inv f$ the set of all relations $R \in \cl{R}_A$ that are preserved by $f$. For a set $\cl{F} \subseteq \cl{O}_A$ of functions, we let $\Inv \cl{F} := \bigcap_{f \in \cl{F}} \Inv f$. The sets $\Inv f$ and $\Inv \cl{F}$ are called the sets of all \emph{invariants} of $f$ and $\cl{F}$, respectively.

By Theorem~\ref{thm:Galois}, $(\Inv, \Pol)$ is the Galois connection induced by the relation $\vartriangleright$ between the set $\cl{O}_A$ of all operations on $A$ and the set $\cl{R}_A$ of all relations on $A$. It was shown by Geiger~\cite{Geiger} and independently by Bodnar\v{c}uk, Kalu\v{z}nin, Kotov and Romov~\cite{BKKR} that for finite sets $A$, the closed subsets of $\cl{O}_A$ under this Galois connection are exactly the clones on $A$. These authors also described the Galois closed subsets of $\cl{R}_A$ by defining an algebra on $\cl{R}_A$ and showing that the closed sets of relations are exactly the \emph{relational clones,} i.e., the subuniverses of the aforementioned algebra on $\cl{R}_A$.

\begin{theorem}[Geiger~\cite{Geiger}; Bodnar\v{c}uk, Kalu\v{z}nin, Kotov and Romov~\cite{BKKR}]
Let $A$ be a finite nonempty set.
\begin{enumerate}[(i)]
\item A set $\cl{F} \subseteq \cl{O}_A$ of operations is the set of polymorphisms of some set $\cl{Q} \subseteq \cl{R}_A$ of relations if and only if $\cl{F}$ is a clone on $A$.
\item A set $\cl{Q} \subseteq \cl{R}_A$ of relations is the set of invariants of some set $\cl{F} \subseteq \cl{O}_A$ of operations if and only if $\cl{Q}$ is a relational clone on $A$.
\end{enumerate}
\end{theorem}

On arbitrary, possibly infinite sets $A$, the Galois closed sets of operations are the locally closed clones, as shown by Szabó~\cite{Szabo} and independently by Pöschel~\cite{Poschel}.  A set $\cl{F} \subseteq \cl{O}_A$ of operations is said to be \emph{locally closed,} if it holds that for all $f \in \cl{O}_A$, say of arity $n$, $f \in \cl{F}$ whenever for all finite subsets $F \subseteq A^n$, there exists a function $g \in \cl{F}^{(n)}$ such that $f|_F = g|_F$.

These results were generalized to iterative algebras (with or without projections) by Harnau~\cite{Harnau1-3} who defined a polarity between operations and relation pairs. An $m$-ary \emph{relation pair} on $A$ is a pair $(R, R')$ where $R, R' \in \cl{R}_A^{(m)}$ for some $m \geq 1$ and $R' \subseteq R$. For $m \geq 1$, denote
\[
\cl{H}_A^{(m)} := \{(R, R') \mid R' \subseteq R \subseteq A^m\}
\]
and
\[
\cl{H}_A = \bigcup_{m \geq 1} \cl{H}_A^{(m)}.
\]
An operation $f \in \cl{O}_A$ is said to \emph{preserve} a relation pair $(R, R') \in \cl{H}_A^{(m)}$, denoted $f \vartriangleright (R, R')$, if for all matrices $\mtrx{M} \in A^{m \times n}$, $\mtrx{M} \prec R$ implies $f \mtrx{M} \in R'$. In light of Theorem~\ref{thm:Galois}, the preservation relation $\vartriangleright$ induces a Galois connection between the sets $\cl{O}_A$ and $\cl{H}_A$. Harnau showed that the closed sets of operations are exactly the universes of iterative algebras. He defined certain operations on the set $\cl{H}_A$ of relation pairs and showed that the Galois closed subsets of relation pairs are precisely the subsets that are closed under these operations.

In~\cite{Lehtonen}, the subalgebras of the reduct $(\cl{O}_A; \zeta, \tau, \nabla, \ast)$ of the full iterative algebra containing all projections were completely characterized in terms of a preservation relation between operations and so-called clusters. In analogy with Harnau's approach to iterative algebras (i.e., considering all subalgebras of $(\cl{O}_A; \zeta, \tau, \Delta, \nabla, \ast)$ with or without projections and thus extending the $\Pol$--$\Inv$ theory of clones and relations), in this paper we relax the closure system to all subalgebras of $(\cl{O}_A; \zeta, \tau, \nabla, \ast)$, not necessarily containing all projections. This is achieved within a Galois framework where the dual objects are systems of pointed multisets. We will also describe the Galois closed sets of these dual objects in terms of explicit closure conditions. Furthermore, we will show that the respective closure systems are uncountable for all $\card{A} \geq 2$.

Such a relaxation is both natural and noteworthy. Motivating examples are those sets of operations obtained from clones by removing all those operations that have arity at most $m$ for some fixed $m \geq 1$. Clearly, these sets are closed under permutation of variables, addition of dummy variables, and composition, and hence they constitute subalgebras of the reduct $(\cl{O}_A; \zeta, \tau, \nabla, \ast)$.

This line of research has been carried out by several authors (see~\cite{BKKR,CF,Geiger,Harnau1-3,Hellerstein,Lehtonen,Pippenger,Poschel,Szabo}; a brief survey of the topic is provided in~\cite{Lehtonen}; see also chapter ``Galois connections for operations and relations'' by R. Pöschel in \cite{DEW}, pp.\ 231--258). Sections~\ref{sec:preli} and~\ref{sec:operations} of the current paper are based on~\cite{CL2010}.

%%%%%%%%%%%%%%%%%%%%%%%%%%%%%%%%%%%%%%%%%%%%%%%

\section{Sets of operations closed under permutation of variables, addition of dummy variables, and composition}
\label{sec:operations}

We now consider the problem of characterizing the sets of operations on an arbitrary nonempty set $A$ that are closed under permutation of variables, addition of dummy variables, and composition (but not necessarily under identification of variables).

A \emph{finite multiset} $S$ on a set $A$ is a map $\nu_S \colon A \to \nat$, called a \emph{multiplicity function,} such that the set $\{x \in A \mid \nu_S(x) \neq 0\}$ is finite. Then the sum $\sum_{x \in A} \nu_S(x)$ is a well-defined natural number, and it is called the \emph{cardinality} of $S$ and denoted by $\card{S}$. The number $\nu_S(x)$ is called the \emph{multiplicity} of $x$ in $S$. We may represent a finite multiset $S$ by giving a list enclosed in set brackets, i.e., $\{a_1, \dotsc, a_n\}$, where each element $x \in A$ occurs $\nu_S(x)$ times. If $S'$ is another multiset on $A$ corresponding to $\nu_{S'} \colon A \to \nat$, then we say that $S'$ is a \emph{submultiset} of $S$, denoted $S' \subseteq S$, if $\nu_{S'}(x) \leq \nu_S(x)$ for all $x \in A$. We denote the set of all finite multisets on $A$ by $\mathcal{M}(A)$. We also denote, for each $p \geq 0$, by $\mathcal{M}^{(p)}(A)$ the set of all finite multisets on $A$ of cardinality at most $p$, i.e., $\mathcal{M}^{(p)}(A) := \{S \in \mathcal{M}(A) \mid \card{S} \leq p\}$.

The set $\mathcal{M}(A)$ is partially ordered by the multiset inclusion relation ``$\subseteq$''. The \emph{join} $S \uplus S'$ and the \emph{difference} $S \setminus S'$ of multisets $S$ and $S'$ are determined by the multiplicity functions
\begin{align*}
\nu_{S \uplus S'}(x) &:= \nu_S(x) + \nu_{S'}(x), \\
\nu_{S \setminus S'}(x) &:= \max \{\nu_S(x) - \nu_{S'}(x), 0\},
\end{align*}
respectively. The \emph{empty multiset} on $A$ is the zero function, and it is denoted by $\varepsilon$. A \emph{partition} of a finite multiset $S$ on $A$ is a multiset $\{S_1, \dotsc, S_n\}$ (on the set of all finite multisets on $A$) of nonempty finite multisets on $A$ such that $S = S_1 \uplus \dotsb \uplus S_n$.

A \emph{pointed multiset} on a set $A$ is a pair $(x, S) \in A \times \mathcal{M}(A)$. The multiset $\{x\} \uplus S$ is called the \emph{underlying multiset} of $(x, S)$. We define the \emph{cardinality} of a pointed multiset $(x, S)$ to be equal to the cardinality of its underlying multiset, i.e., $\card{(x, S)} := \card{\{x\} \uplus S} = \card{S} + 1$. If $(x, S)$ and $(x', S)$ are pointed multisets on $A$, then we say that $(x, S)$ is a \emph{pointed submultiset} of $(x', S')$, denoted $(x, S) \subseteq (x', S')$, if $x = x'$ and $S \subseteq S'$.

For an $m \times n$ matrix $\mtrx{M} \in A^{m \times n}$, the \emph{multiset of columns} of $\mtrx{M}$ is the multiset $\mtrx{M}^*$ on $A^m$ defined by the function $\chi_\mtrx{M}$ which maps each $m$-tuple $\vect{a} \in A^m$ to the number of times $\vect{a}$ occurs as a column of $\mtrx{M}$. A matrix $\mtrx{N} \in A^{m \times n'}$ is a \emph{submatrix} of $\mtrx{M} \in A^{m \times n}$ if $\mtrx{N}^* \subseteq \mtrx{M}^*$, i.e., $\chi_\mtrx{N}(\vect{a}) \leq \chi_\mtrx{M}(\vect{a})$ for all $\vect{a} \in A^m$.

For an integer $m \geq 1$, an $m$-ary \emph{system of pointed multisets} on $A$ (a \emph{system} for short) is a pair $(\Phi, \Phi') \in \mathcal{P}(\mathcal{M}(A^m)) \times \mathcal{P}(A^m \times \mathcal{M}(A^m))$, where the \emph{antecedent} $\Phi \subseteq \mathcal{M}(A^m)$ is a set of finite multisets on $A^m$ and the \emph{consequent} $\Phi' \subseteq A^m \times \mathcal{M}(A^m)$ is a set of pointed multisets on $A^m$, satisfying the following two conditions:
\begin{enumerate}[(1)]
\item $(x,S) \in \Phi'$ and $S' \subseteq S$ imply $(x,S') \in \Phi'$;
\item $(x,S) \in \Phi'$ implies $\{x\} \uplus S \in \Phi$.
\end{enumerate}
For $m \geq 1$, we denote the set of all $m$-ary systems of pointed multisets on $A$ by $\cl{W}_A^{(m)}$, and we denote the set of all systems of pointed multisets on $A$ by
\[
\cl{W}_A := \bigcup_{m \geq 1} \cl{W}_A^{(m)}.
\]

For $(\Phi, \Phi'), (\Psi, \Psi') \in \mathcal{P}(\mathcal{M}(A^m)) \times \mathcal{P}(A^m \times \mathcal{M}(A^m))$, we write $(\Phi, \Phi') \subseteq (\Psi, \Psi')$ to mean componentwise inclusion, i.e., $\Phi \subseteq \Psi$ and $\Phi' \subseteq \Psi'$.

If $\mtrx{M} \in A^{m \times n}$ and $\Phi$ is a set of multisets over $A^m$, we write $\mtrx{M} \prec \Phi$ to mean that the multiset $\mtrx{M}^*$ of columns of $\mtrx{M}$ is an element of $\Phi$. If $\mtrx{M} = (\vect{m}_1, \dots, \vect{m}_n) \in A^{m \times n}$, $n \geq 1$, and $\Phi'$ is a set of pointed multisets $(x, S) \in A^m \times \mathcal{M}(A^m)$, then we write $\mtrx{M} \prec \Phi'$ to mean that $(\vect{m}_1, \{\vect{m}_2, \dots, \vect{m}_n\}) \in \Phi'$. If $f \in \cl{O}_A^{n}$ and $(\Phi, \Phi') \in \cl{W}_A^{(m)}$, we say that $f$ \emph{preserves} $(\Phi, \Phi')$, denoted $f \vartriangleright (\Phi, \Phi')$, if for every matrix $\mtrx{M} \in A^{m \times p}$ for some $p \geq 0$, it holds that whenever $\mtrx{M} \prec \Phi$ and $\mtrx{M} = [\mtrx{M}_1 | \mtrx{M}_2]$ where $\mtrx{M}_1$ has $n$ columns and $\mtrx{M}_2$ may be empty, we have that $[f \mtrx{M}_1 | \mtrx{M_2}] \prec \Phi'$.

In light of Theorem~\ref{thm:Galois}, the relation $\vartriangleright$ establishes a Galois connection between the sets $\cl{O}_A$ and $\cl{W}_A$. We say that a set $\cl{F} \subseteq \cl{O}_A$ of operations on $A$ is \emph{characterized} by a set $\cl{W} \subseteq \cl{W}_A$ of systems of pointed multisets, if
\[
\cl{F} = \{f \in \cl{O}_A \mid \forall (\Phi, \Phi') \in \cl{W} \colon f \vartriangleright (\Phi, \Phi')\},
\]
i.e., $\cl{F}$ is precisely the set of operations on $A$ that preserve every system of pointed multisets in $\cl{W}$. Similarly, we say that $\cl{W}$ is \emph{characterized} by $\cl{F}$, if
\[
\cl{W} = \{(\Phi, \Phi') \in \cl{W}_A \mid \forall f \in \cl{F} \colon f \vartriangleright (\Phi, \Phi')\},
\]
i.e., $\cl{W}$ is precisely the set of systems of pointed multisets that are preserved by every operation in $\cl{F}$. Thus, the Galois closed sets of operations (systems of pointed multisets) are exactly those that are characterized by systems of pointed multisets (operations, respectively).

\begin{example}
Let $p \geq 1$ be an integer and consider the set $\cl{O}_A^{(\geq p)} := \bigcup_{n \geq p} \cl{O}_A^{(n)}$ of operations on $A$ of arity at least $p$. It is clearly a subalgebra of $(\cl{O}_A; \zeta, \tau, \nabla, \ast)$ which does not contain all projections. Let $S$ be any multiset on $A$ of cardinality $p-1$. It is easy to verify that the system $(\{S\}, \emptyset) \in \cl{W}_A^{(1)}$ characterizes $\cl{O}_A^{(\geq p)}$. For, every operation on $A$ of arity at least $p$ vacuously preserves $(\{S\}, \emptyset)$, but no operation of arity less than $p$ can preserve $(\{S\}, \emptyset)$.
\end{example}

\begin{example}
Let $R \subseteq A^m$ be an $m$-ary relation $R$ on $A$. Let $\Phi_R$ be the set of all finite multisets $S$ on $A^m$ such that $\nu_S(\vect{a}) \neq 0$ only if $\vect{a} \in R$, and let $\Phi'_R := R \times \Phi_R$. It is easy to verify that $f \vartriangleright R$ if and only if $f \vartriangleright (\Phi_R, \Phi'_R)$. Hence, if $\cl{C}$ is a clone on $A$ such that $\cl{C} = \Pol \cl{Q}$ for some set $\cl{Q}$ of relations on $A$, then $\cl{C}$ is characterized by the set $\{(\Phi_R, \Phi'_R) \mid R \in \cl{Q}\}$ of systems of pointed multisets.

As a corollary, for a clone $\cl{C} = \Pol \cl{Q}$ and an integer $p \geq 1$, the set $\cl{C}^{(\geq p)} := \cl{C} \cap \cl{O}_A^{(\geq p)}$ of at least $p$-ary members of $\cl{C}$ is characterized by the set $\{(\Phi_R, \Phi'_R) \mid R \in \cl{Q}\} \cup \{(\{S\}, \emptyset)\}$ of systems of pointed multisets, where $S$ is any multiset on $A$ with $\card{S} = p-1$.
\end{example}

\begin{lemma}
\label{lem:partitions}
Let $\cl{F} \subseteq \cl{O}_A$ be a locally closed set of operations that is closed under permutation of variables, addition of dummy variables, and composition. Then for every $g \in \cl{O}_A \setminus \cl{F}$, there exists a system $(\Phi, \Phi') \in \cl{W}_A$ that is preserved by every operation in $\cl{F}$ but not by $g$.
\end{lemma}

\begin{proof}
The statement is vacuously true for $\cl{F} = \cl{O}_A$, and it is easily seen to hold for $\cl{F} = \emptyset$. Thus, we can assume that $\emptyset \subsetneq \cl{F} \subsetneq \cl{O}_A$. Suppose that $g \in \cl{O}_A \setminus \cl{F}$ is $n$-ary. Since $\cl{F}$ is locally closed, there is a finite subset $F \subseteq A^n$ such that $g|_F \neq f|_F$ for every $f \in \cl{F}^{(n)}$. Clearly $F$ is nonempty. Let $\mtrx{M}$ be a $\card{F} \times n$ matrix whose rows are the elements of $F$ in some fixed order.

Let $\mu$ be the smallest integer such that $\cl{F}^{(\mu)} \neq \emptyset$. Note that since $\cl{F}$ is closed under addition of dummy variables, $\cl{F}^{(m)} \neq \emptyset$ for all $m \geq \mu$. Let $X$ be any submultiset of $\mtrx{M}^*$. (Recall that $\mtrx{M}^*$ denotes the multiset of columns of $\mtrx{M}$.) Let $\Pi := (\mtrx{M}_1, \dotsc, \mtrx{M}_q)$ be a sequence of submatrices of $\mtrx{M}$ such that $\{\mtrx{M}_1^*, \dotsc, \mtrx{M}_q^*\}$ is a partition of $\mtrx{M}^* \setminus X$ where each block $\mtrx{M}_i^*$ has cardinality at least $\mu$. For $1 \leq i \leq q$, let $\vect{d}^i \in \cl{F} \mtrx{M}_i$, and let $\mtrx{D} := (\vect{d}^1, \dotsc, \vect{d}^q)$. (Note that each $\cl{F} \mtrx{M}_i$ is nonempty, because $\cl{F}$ contains functions of arity $\lvert \mtrx{M}_i^* \rvert \geq \mu$. Observe also that if $\mtrx{M}_i$ is a submatrix of $\mtrx{M}_j$, then $\cl{F} \mtrx{M}_i \subseteq \cl{F} \mtrx{M}_j$, because $\cl{F}$ is closed under permutation of variables and addition of dummy variables; in particular, each $\cl{F} \mtrx{M}_i$ is a subset of $\cl{F} \mtrx{M}$.) Denote $\lceil X, \Pi, \mtrx{D} \rfloor := \mtrx{D}^* \uplus X$, and for $X \neq \mtrx{M}^*$, denote
\[
\langle X, \Pi, \mtrx{D} \rangle := (\vect{d}^1, \{\vect{d}^2, \dots, \vect{d}^q\} \uplus X) \in A^m \times \mathcal{M}(A^m).
\]
Note that if $\langle X, \Pi, \mtrx{D} \rangle = (x, S)$, then $\lceil X, \Pi, \mtrx{D} \rfloor = \{x\} \uplus S$.

We define $\Phi'$ to be the set of all $\langle X, \Pi, \mtrx{D} \rangle$ for all possible choices of $X$, $\Pi$, and $\mtrx{D}$ such that $X \neq \mtrx{M}^*$, and we define $\Phi$ to be the set of all $\lceil X, \Pi, \mtrx{D} \rfloor$ for all possible choices of $X$, $\Pi$ and $\mtrx{D}$, i.e.,
\[
\Phi := \{x \uplus S \mid (x,S) \in \Phi'\} \cup \{\mtrx{M}^*\}.
\]
We first verify that $(\Phi, \Phi')$ is indeed a system of pointed multisets. The first condition in the definition of a system of pointed multisets is clearly satisfied, by the definition of $(\Phi, \Phi')$. For the second condition, let $(x,S) = \langle X, \Pi, \mtrx{D} \rangle \in \Phi'$, where $\Pi := (\mtrx{M}_1, \dots, \mtrx{M}_q)$ and $\mtrx{D} := (\vect{d}^1, \dots, \vect{d}^q)$. A simple inductive argument proves that $(x,S') \in \Phi'$ for all $S' \subseteq S$, and it suffices to show that $(x,S') \in \Phi'$ whenever $S = S' \uplus \{\vect{y}\}$ for some $\vect{y} \in A^m$. If $\vect{y} \in X$, then we let 
\begin{align*}
X' &:= X \setminus \{\vect{y}\}, \\
\mtrx{M}'_1 &:= [M_1 | \vect{y}], \\
\Pi' &:= (\mtrx{M}'_1, \mtrx{M}_2, \dots, \mtrx{M}_q).
\end{align*}
Since $\cl{F}$ is closed under addition of dummy variables, we have that $\vect{d}^1 \in \cl{F} \mtrx{M}'_1$, and hence we have $(x, S') = \langle X', \Pi', \mtrx{D} \rangle \in \Phi'$. If $\vect{y} \in \{\vect{d}^2, \dots, \vect{d}^q\}$, say, $\vect{y} = \vect{d}^j$, then we let
\begin{align*}
\mtrx{M}'_1 &:= [\mtrx{M}_1 | \mtrx{M}_j], \\
\Pi' &:= (\mtrx{M}_1, \dots, \mtrx{M}_{j-1}, \mtrx{M}_{j+1}, \dots, \mtrx{M}_q),
\end{align*}
and again, since $\cl{F}$ is closed under addition of dummy variables, we have that $\vect{d}^1 \in \cl{F} \mtrx{M}'_1$ and we can let
\[
\mtrx{D}' := (\vect{d}^1, \dots, \vect{d}^{j-1}, \vect{d}^{j+1}, \dots, \vect{d}^q),
\]
and we have $(x, S') = \langle X, \Pi', \mtrx{D}' \rangle \in \Phi'$.

Observe first that $g \not\vartriangleright (\Phi, \Phi')$. For, we have that $\mtrx{M} \prec \Phi$ by the definition of $\Phi$. On the other hand, since $g \mtrx{M} \notin \cl{F} \mtrx{M}$, we have that $g \mtrx{M}$ is not the first component of $\langle X, \Pi, \mtrx{D} \rangle$ for all $X$, $\Pi$, $\mtrx{D}$, and hence $g \mtrx{M} \nprec \Phi'$.

It remains to show that $f \vartriangleright (\Phi, \Phi')$ for all $f \in \cl{F}$. Assume that $f$ is $n$-ary. If $\mtrx{N} := [\mtrx{N}_1 | \mtrx{N}_2] \prec \Phi$, where $\mtrx{N}_1$ has $n$ columns, then $\mtrx{N}^* = \lceil X, \Pi, \mtrx{D} \rfloor$ for some
\[
X,
\quad
\Pi := (\mtrx{M}_1, \dotsc, \mtrx{M}_q),
\quad
\mtrx{D} := (\vect{d}^1, \dotsc, \vect{d}^q),
\]
where $\{\mtrx{M}_1^*, \dotsc, \mtrx{M}_q^*\}$ is a partition of $\mtrx{M}^* \setminus X$ and for $1 \leq i \leq q$, $\vect{d}^i \in \cl{F} \mtrx{M}_i$. We will show by induction on $q$ that for every $f \in \cl{F}$, there exist $X'$, $\Pi'$, $\mtrx{D}'$ such that $(f \mtrx{N}_1, \mtrx{N}_2^*) = \langle X', \Pi', \mtrx{D}' \rangle$ and hence $[f \mtrx{N}_1 | \mtrx{N}_2] \prec \Phi'$.

If $q = 0$, then we have $X = \mtrx{M}^*$, $\Pi = ()$, $\mtrx{D} = ()$, and $\lceil X, \Pi, \mtrx{D} \rfloor = \mtrx{M}^*$, and the condition $\mtrx{N}^* = \mtrx{M}^*$ implies that $\mtrx{N}_1$ is a submatrix of $\mtrx{M}$. Then $f \mtrx{N}_1 \in \cl{F} \mtrx{N}_1$ and
\[
(f \mtrx{N}_1, \mtrx{N}_2^*) = \langle \mtrx{M}^* \setminus \mtrx{N}_1^*, (\mtrx{N}_1), (f \mtrx{N}_1) \rangle.
\]

Assume that the claim holds for $q = k \geq 0$, and consider the case that $q = k + 1$. We assume that $\mtrx{N} = [\mtrx{N}_1 | \mtrx{N}_2]$ and $\mtrx{N}^* = \lceil X, \Pi, \mtrx{D} \rfloor$. If $\mtrx{N}_1^* \subseteq X$, then $f \mtrx{N}_1 \in \cl{F} \mtrx{N}_1$ and
\[
(f \mtrx{N}_1, \mtrx{N}_2^*) =
\langle X \setminus \mtrx{N}_1^*, (\mtrx{M}_1, \dotsc, \mtrx{M}_{k+1}, \mtrx{N}_1), (\vect{d}^1, \dotsc, \vect{d}^{k+1}, f \mtrx{N}_1) \rangle.
\]
Otherwise, for some $i \in \{1, \dotsc, k+1\}$, $\vect{d}^i$ is a column of $\mtrx{N}_1$. Denote by $\mtrx{N}'_1$ the matrix obtained from $\mtrx{N}_1$ by deleting the column $\vect{d}^i$. Since $\cl{F}$ is closed under permutation of variables, there is an operation $f' \in \cl{F}^{(n)}$ such that $f \mtrx{N}_1 = f' [\vect{d}^i | \mtrx{N}'_1]$. By the definition of $\vect{d}^i$, there is an operation $h \in \cl{F}$ such that $h \mtrx{M}_i = \vect{d}^i$, and we have that
\begin{equation}
\label{eq:fh}
f' [\vect{d}^i | \mtrx{N}'_1]
= f' [h \mtrx{M}_i | \mtrx{N}'_1]
= (f' \ast h) [\mtrx{M}_i | \mtrx{N}'_1].
\end{equation}
Since $\cl{F}$ is closed under composition, $f' \ast h \in \cl{F}$. Furthermore,
\begin{multline*}
[\mtrx{M}_i | \mtrx{N}'_1 | \mtrx{N}_2]^* = \\
\lceil X \uplus \mtrx{M}_i^*,
(\mtrx{M}_1, \dotsc, \mtrx{M}_{i-1}, \mtrx{M}_{i+1}, \dotsc, \mtrx{M}_{k+1}),
(\vect{d}^1, \dotsc, \vect{d}^{i-1}, \vect{d}^{i+1}, \dotsc, \vect{d}^{k+1}) \rfloor.
\end{multline*}
By the induction hypothesis, there exist $X'$, $\Pi'$, $\mtrx{D}'$ such that
\[
((f' \ast h) [\mtrx{M}_i | \mtrx{N}'_1], \mtrx{N}_2^*) = \langle X', \Pi', \mtrx{D}' \rangle.
\]
By \eqref{eq:fh}, we have that $(f' \ast h) [\mtrx{M}_i | \mtrx{N}'_1] = f \mtrx{N}_1$, and hence $(f \mtrx{N}_1, \mtrx{N}_2^*) = \langle X', \Pi', \mtrx{D}' \rangle$.
\end{proof}

\begin{theorem}
\label{thm:pac}
Let $A$ be an arbitrary, possibly infinite nonempty set. For any set $\cl{F} \subseteq \cl{O}_A$ of operations, the following two conditions are equivalent:
\begin{enumerate}[(i)]
\item $\cl{F}$ is locally closed and closed under permutation of variables, addition of dummy variables, and composition.
\item $\cl{F}$ is characterized by a set $\cl{W} \subseteq \cl{W}_A$ of systems of pointed multisets.
\end{enumerate}
\end{theorem}

\begin{proof}
$\text{(ii)} \Rightarrow \text{(i)}$: It is straightforward to verify that the set of operations preserving a set of systems of pointed multisets is closed under permutation of variables and addition of dummy variables. To see that it is closed under composition, let $f \in \cl{F}^{(n)}$ and $g \in \cl{F}^{(p)}$, and consider $f \ast g \colon A^{n + p - 1} \to A$. Let $(\Phi, \Phi') \in \cl{W}$, and let $\mtrx{M} := [\mtrx{M}_1 | \mtrx{M}_2 | \mtrx{M}_3] \prec \Phi$, where $\mtrx{M}_1$ has $p$ columns and $\mtrx{M}_2$ has $n-1$ columns. Since $g \vartriangleright (\Phi, \Phi')$, we have that $[g \mtrx{M}_1 | \mtrx{M}_2 | \mtrx{M}_3] \prec \Phi'$ and hence, by property (2) of the definition of a system of pointed multisets, we have $[g \mtrx{M}_1 | \mtrx{M}_2 | \mtrx{M}_3] \prec \Phi$. Then $[g \mtrx{M}_1 | \mtrx{M}_2]$ has $n$ columns, and since $f \vartriangleright (\Phi, \Phi')$, we have that $\bigl[ f[g \mtrx{M}_1 | \mtrx{M}_2] \big| \mtrx{M}_3 \bigr] \prec \Phi'$. But
\[
f[g \mtrx{M}_1 | \mtrx{M}_2] = (f \ast g)[\mtrx{M}_1 | \mtrx{M}_2],
\]
so $\bigl[ (f \ast g)[\mtrx{M}_1 | \mtrx{M}_2] \big| \mtrx{M}_3 \bigr] \prec \Phi'$, and thus $f \ast g \vartriangleright (\Phi, \Phi')$.

It remains to show that $\cl{F}$ is locally closed. It is clear that $\cl{O}_A$ is locally closed, so we may assume that $\cl{F} \neq \cl{O}_A$.
Suppose on the contrary that there is a $g \in \cl{O}_A \setminus \cl{F}$, say of arity $n$, such that for every finite subset $F \subseteq A^n$, there is an $f \in \cl{F}^{(n)}$ such that $g|_F = f|_F$. Since $\cl{F}$ is characterized by $\cl{W}$ and $g \notin \cl{F}$, there is a system $(\Phi, \Phi') \in \cl{W}$ such that $g \not\vartriangleright (\Phi, \Phi')$, and hence for some matrix $\mtrx{M} := [\mtrx{M}_1 | \mtrx{M}_2] \prec \Phi$ where $\mtrx{M}_1$ has $n$ columns, we have that $[g \mtrx{M}_1 | \mtrx{M}_2] \nprec \Phi'$. Let $F$ be the finite set of rows of $\mtrx{M}_1$. By our assumption, there is an $f \in \cl{F}^{(n)}$ such that $g|_F = f|_F$, and hence
\[
f \mtrx{M}_1 = f|_F \mtrx{M}_1 = g|_F \mtrx{M}_1 = g \mtrx{M}_1,
\]
and so $[f \mtrx{M}_1 | \mtrx{M}_2] \nprec \Phi'$, which contradicts the assumption that $f \vartriangleright (\Phi, \Phi')$.

\smallskip
$\text{(i)} \Rightarrow \text{(ii)}$: It follows from Lemma~\ref{lem:partitions} that for every operation $g \in \cl{O}_A \setminus \cl{F}$, there exists a system $(\Phi, \Phi') \in \cl{W}_A$ that is preserved by every operation in $\cl{F}$ but not by $g$. The set of all such ``separating'' systems of pointed multisets, for each $g \in \cl{O}_A \setminus \cl{F}$, characterizes $\cl{F}$.
\end{proof}

%%%%%%%%%%%%%%%%%%%%%%%%%%%%%%%%%%%%%%%%%%%%%%%

\section{Closure conditions for systems of pointed multisets}
\label{sec:cl}

In this section we will describe the sets of systems of pointed multisets that are characterized by sets of operations in terms of explicit closure conditions. We will follow Couceiro and Foldes's~\cite{CF} proof techniques and adapt their notion of conjunctive minor to systems of pointed multisets. We first introduce several technical notions and definitions that will be needed in the statement of Theorem~\ref{thm:spms} and in its proof.

For maps $f \colon A \to B$ and $g \colon C \to D$, the composition $g \circ f$ is defined only if $B = C$. Removing this restriction, the \emph{concatenation} of $f$ and $g$ is defined to be the map $gf \colon f^{-1}[B \cap C] \to D$ given by the rule $(gf)(a) = g \bigl( f(a) \bigr)$ for all $a \in f^{-1}[B \cap C]$. Clearly, if $B = C$, then $gf = g \circ f$; thus functional composition is subsumed and extended by concatenation. Concatenation is associative, i.e., for any maps $f$, $g$, $h$, we have $h(gf) = (hg)f$.

For a family $(g_i)_{i \in I}$ of maps $g_i \colon A_i \to B_i$ such that $A_i \cap A_j = \emptyset$ whenever $i \neq j$, we define the (\emph{piecewise}) \emph{sum of the family $(g_i)_{i \in I}$} to be the map $\sum_{i \in I} g_i \colon \bigcup_{i \in I} A_i \to \bigcup_{i \in I} B_i$ whose restriction to each $A_i$ coincides with $g_i$. If $I$ is a two-element set, say $I = \{1, 2\}$, then we write $g_1 + g_2$. Clearly, this operation is associative and commutative.

Concatenation is distributive over summation, i.e., for any family $(g_i)_{i \in I}$ of maps on disjoint domains and any map $f$,
\[
\bigl( \sum_{i \in I} g_i \bigr) f = \sum_{i \in I} (g_i f)
\qquad \text{and} \qquad
f \bigl( \sum_{i \in I} g_i \bigr) = \sum_{i \in I} (f g_i).
\]
In particular, if $g_1$ and $g_2$ are maps with disjoint domains, then
\[
(g_1 + g_2) f = (g_1 f) + (g_2 f)
\qquad \text{and} \qquad
f (g_1 + g_2) = (f g_1) + (f g_2).
\]

Let $g_1, \dotsc, g_n$ be maps from $A$ to $B$. The $n$-tuple $(g_1, \dotsc, g_n)$ determines a \emph{vector-valued map} $g \colon A \to B^n$, given by $g(a) := \bigl( g_1(a), \dotsc, g_n(a) \bigr)$ for every $a \in A$. For $f \colon B^n \to C$, the composition $f \circ g$ is a map from $A$ to $C$, denoted by $f(g_1, \dotsc, g_n)$, and called the \emph{composition of $f$ with $g_1, \dotsc, g_n$.} Suppose that $A \cap A' = \emptyset$ and $g'_1, \dotsc, g'_n$ are maps from $A'$ to $B$. Let $g$ and $g'$ be the vector-valued maps determined by $(g_1, \dotsc, g_n)$ and $(g'_1, \dotsc, g'_n)$, respectively. We have that $f(g + g') = (fg) + (fg')$, i.e.,
\[
f \bigl( (g_1 + g'_1), \dotsc, (g_n + g'_n) \bigr) = f(g_1, \dotsc, g_n) + f(g'_1, \dotsc, g'_n).
\]

For $B \subseteq A$, $\iota_{A B}$ denotes the canonical injection (inclusion map) from $B$ to $A$. Thus the \emph{restriction} $f|_B$ of any map $f \colon A \to C$ to the subset $B$ is given by $f|_B = f \iota_{A B}$.

\begin{remark}
Observe that the notation $f \mtrx{M}$ introduced in Section~\ref{sec:preli} is in accordance with the notation for concatenation of mappings. Since a matrix $\mtrx{M} := (\vect{a}^1, \dotsc, \vect{a}^n)$ is an $n$-tuple of $m$-tuples $\vect{a}^i \colon m \to A$, $1 \leq i \leq n$, the composition of the vector-valued map $(\vect{a}^1, \dotsc, \vect{a}^n) \colon m \to A^n$ with $f \colon A^n \to B$ gives rise to the $m$-tuple $f(\vect{a}^1, \dotsc, \vect{a}^n) \colon m \to B$.
\end{remark}

Let $m$ and $n$ be positive integers (viewed as ordinals, i.e., $m = \{0, \dotsc, m-1\}$). Let $h \colon n \to m \cup V$ where $V$ is an arbitrary set of symbols disjoint from the ordinals, called \emph{existentially quantified indeterminate indices,} or simply \emph{indeterminates,} and let $\sigma \colon V \to A$ be any map, called a \emph{Skolem map.} Then each $m$-tuple $\vect{a} \in A^m$, being a map $\vect{a} \colon m \to A$, gives rise to an $n$-tuple $(\vect{a} + \sigma) h =: (b_0, \dotsc, b_{n-1}) \in A^n$, where
\[
b_i :=
\begin{cases}
a_{h(i)}, & \text{if $h(i) \in \{0, 1, \dotsc, m-1\}$,} \\
\sigma(h(i)), & \text{if $h(i) \in V$.}
\end{cases}
\]
Let $H := (h_j)_{j \in J}$ be a nonempty family of maps $h_j \colon n_j \to m \cup V$, where each $n_j$ is a positive integer. Then $H$ is called a \emph{minor formation scheme} with \emph{target} $m$, \emph{indeterminate set} $V$, and \emph{source family} $(n_j)_{j \in J}$.

Let $(\Phi_j)_{j \in J}$ be a family of sets of multisets (or pointed multisets), each $\Phi_j$ on the set $A^{n_j}$, and let $\Phi$ be a set of multisets (or pointed multisets, respectively), on $A^m$. We say that $\Phi$ is a \emph{restrictive conjunctive minor} of the family $(\Phi_j)_{j \in J}$ \emph{via} $H$, if, for every $m \times n$ matrix $\mtrx{M} := (\vect{a}^1, \dotsc, \vect{a}^n) \in A^{m \times n}$,
\[
\mtrx{M} \prec \Phi
\implies
\bigl[ \exists \sigma_1, \dotsc, \sigma_n \in A^V \: \forall j \in J \colon \bigl( (\vect{a}^1 + \sigma_1) h_j, \dotsc, (\vect{a}^n + \sigma_n) h_j \bigr) \prec \Phi_j \bigr].
\]
On the other hand, if, for every $m \times n$ matrix $\mtrx{M} := (\vect{a}^1, \dotsc, \vect{a}^n) \in A^{m \times n}$,
\[
\bigl[ \exists \sigma_1, \dotsc, \sigma_n \in A^V \: \forall j \in J \colon \bigl( (\vect{a}^1 + \sigma_1) h_j, \dotsc, (\vect{a}^n + \sigma_n) h_j \bigr) \prec \Phi_j \bigr]
\implies
\mtrx{M} \prec \Phi,
\]
then we say that $\Phi$ is an \emph{extensive conjunctive minor} of the family $(\Phi_j)_{j \in J}$ \emph{via} $H$. If $\Phi$ is both a restrictive conjunctive minor and an extensive conjunctive minor of the family $(\Phi_j)_{j \in J}$ via $H$, i.e., for every $m \times n$ matrix $\mtrx{M} := (\vect{a}^1, \dotsc, \vect{a}^n) \in A^{m \times n}$,
\[
\mtrx{M} \prec \Phi
\iff
\bigl[ \exists \sigma_1, \dotsc, \sigma_n \in A^V \: \forall j \in J \colon \bigl( (\vect{a}^1 + \sigma_1) h_j, \dotsc, (\vect{a}^n + \sigma_n) h_j \bigr) \prec \Phi_j \bigr],
\]
then $\Phi$ is said to be a \emph{tight conjunctive minor} of the family $(\Phi_j)_{j \in J}$ \emph{via} $H$.

If $(\Phi, \Phi') \in \cl{W}_A$ is a system of pointed multisets on $A$ and $(\Phi_j, \Phi'_j)_{j \in J}$ is a family of systems of pointed multisets on $A$ (of various arities) such that $\Phi$ is a restrictive conjunctive minor of the family $(\Phi_j)_{j \in J}$ of multisets via a scheme $H$ and $\Phi'$ is an extensive conjunctive minor of the family $(\Phi'_j)_{j \in J}$ of pointed multisets via the same scheme $H$, then $(\Phi, \Phi')$ is said to be a \emph{conjunctive minor} of the family $(\Phi_j, \Phi'_j)_{j \in J}$ \emph{via} $H$. If both $\Phi$ and $\Phi'$ are tight conjunctive minors of the respective families via $H$, then $(\Phi, \Phi')$ is said to be a \emph{tight conjunctive minor} of the family $(\Phi_j, \Phi'_j)_{j \in J}$ \emph{via} $H$.
% Tight conjunctive minors of families of generalized constraints are not unique, but if both $(\phi, S)$ and $(\phi', S')$ are tight conjunctive minors of the same family of generalized constraints via the same scheme, then $S = S'$.
If the minor formation scheme $H := (h_j)_{j \in J}$ and the family $(\Phi_j, \Phi'_j)_{j \in J}$ are indexed by a singleton $J := \{0\}$, then a tight conjunctive minor $(\Phi, \Phi')$ of a family consisting of a single system of pointed multisets $(\Phi_0, \Phi'_0)$ is called a \emph{simple minor} of $(\Phi_0, \Phi'_0)$.

\begin{lemma}
\label{lemma:cm}
Let $(\Phi, \Phi')$ be a conjunctive minor of a nonempty family \linebreak[4] $(\Phi_j, \Phi'_j)_{j \in J}$ of members of $\cl{W}_A$, and let $f \in \cl{O}_A$. If $f \vartriangleright (\Phi_j, \Phi'_j)$ for all $j \in J$, then $f \vartriangleright (\Phi, \Phi')$.
\end{lemma}

\begin{proof}
Let $(\Phi, \Phi')$ be an $m$-ary conjunctive minor of the family $(\Phi_j, \Phi'_j)_{j \in J}$ via the scheme $H := (h_j)_{j \in J}$, $h_j \colon n_j \to m \cup V$. Let $\mtrx{M} := (\vect{a}^1, \dotsc, \vect{a}^n)$ be an arbitrary $m \times n$ matrix such that $\mtrx{M} \prec \Phi$. We want to prove that $f \mtrx{M} \prec \Phi'$. Since $\Phi$ is a restrictive conjunctive minor of $(\Phi_j)_{j \in J}$ via $H = (h_j)_{j \in J}$, there exist Skolem maps $\sigma_i \colon V \to A$, $1 \leq i \leq n$, such that for every $j \in J$, $\mtrx{M}_j := \bigl( (\vect{a}^1 + \sigma_1) h_j, \dotsc, (\vect{a}^n + \sigma_n) h_j \bigr) \prec \Phi_j$.

Since $\Phi'$ is an extensive conjunctive minor of $(\Phi'_j)_{j \in J}$ via the same scheme $H = (h_j)_{j \in J}$, to prove that $f \mtrx{M} \prec \Phi'$, it suffices to give a Skolem map $\sigma \colon V \to A$ such that, for all $j \in J$, $(f \mtrx{M} + \sigma) h_j \prec \Phi'_j$. Let $\sigma := f(\sigma_1, \dotsc, \sigma_n)$. We have that, for each $j \in J$,
\begin{align*}
(f \mtrx{M} + \sigma) h_j
&= \bigl( f(\vect{a}^1, \dotsc, \vect{a}^n) + f(\sigma_1, \dotsc, \sigma_n) \bigr) h_j \\
&= \bigl( f(\vect{a}^1 + \sigma_1, \dotsc, \vect{a}^n + \sigma_n) \bigr) h_j \\
&= f \bigl( (\vect{a}^1 + \sigma_1) h_j, \dotsc, (\vect{a}^n + \sigma_n) h_j \bigr)
= f \mtrx{M}_j.
\end{align*}
By our assumption $f \vartriangleright (\Phi_j, \Phi'_j)$, so we have $f \mtrx{M}_j \prec \Phi'_j$.
\end{proof}

We say that a set $\cl{W} \subseteq \cl{W}_A$ of systems of pointed multisets is \emph{closed under formation of conjunctive minors} if whenever $(\Phi_j, \Phi'_j)_{j \in J}$ is a nonempty family of members of $\cl{W}$, all conjunctive minors of the family $(\Phi_j, \Phi'_j)_{j \in J}$ are also in $\cl{W}$.

Let $(\Phi, \Phi'), (\Psi, \Psi') \in \cl{W}_A^{(m)}$. If $\Phi \subseteq \Psi$ and $\Phi' = \Psi'$, then we say that $(\Phi, \Phi')$ is obtained from $(\Psi, \Psi')$ by \emph{restricting the antecedent.} If $\Phi = \Psi$ and $\Phi' \supseteq \Psi'$, then we say that $(\Phi, \Phi')$ is obtained from $(\Psi, \Psi')$ by \emph{extending the consequent.}
The formation of conjunctive minors subsumes the formation of simple minors as well as the operations of restricting the antecedent and extending the consequent. Simple minors in turn subsume permutation of arguments, projection, identification of arguments, and addition of a dummy argument, operations which can be defined for systems of pointed multisets in an analogous way as for Pippenger's~\cite{Pippenger} constraints or Hellerstein's~\cite{Hellerstein} generalized constraints.

The $m$-ary \emph{trivial system of pointed multisets on $A$} is $\Omega_m := (\mathcal{M}(A^m), A^m \times \mathcal{M}(A^m))$.
For $p \geq 0$, the $m$-ary \emph{trivial system of pointed multisets on $A$ of breadth $p$} is $\Omega_m^{(p)} := (\mathcal{M}^{(p)}(A^m), A^m \times \mathcal{M}^{(p-1)}(A^m))$.
The $m$-ary \emph{empty system} on $A$ is the pair $(\emptyset, \emptyset)$. Note that $\Omega_m^{(0)} \neq \emptyset$, because $(\{\varepsilon\}, \emptyset)$ is the unique member of $\Omega_m^{(0)}$.
The $m$-ary \emph{equality system} on $A$, denoted $\mathbf{E}_m$, is the system $\mathbf{E}_m := (E_m, E'_m)$, where
\begin{align*}
E_m & := \{S \in \mathcal{M}(A^m) \mid \nu_S(a_1, \dots, a_m) \neq 0 \implies a_1 = \dots = a_m\}, \\
E'_m & := \{(a, \dots, a) \in A^m \mid a \in A\} \times E_m.
\end{align*}

\begin{lemma}
\label{lemma:all}
Let $\cl{W} \subseteq \cl{W}_A$ be a set of systems of pointed multisets that contains the binary equality system and the unary empty system. If $\cl{W}$ is closed under formation of conjunctive minors, then it contains all trivial systems, all equality systems, and all empty systems.
\end{lemma}

\begin{proof}
The unary trivial system is a simple minor of the binary equality system via the scheme $H := \{h\}$, where $h \colon 2 \to 1$ is given by $h(0) = h(1) = 0$ (by identification of arguments). The $m$-ary trivial system is a simple minor of the unary trivial system via the scheme $H := \{h\}$, where $h \colon 1 \to m$ is given by $h(0) = 0$ (by addition of $m-1$ dummy arguments).

For $m \geq 2$, the $m$-ary equality system is a conjunctive minor of the binary equality system via the scheme $H := (h_i)_{i \in m - 1}$, where $h_i \colon 2 \to m$ is given by $h_i(0) = i$, $h_i(1) = i + 1$ (by addition of $n-2$ dummy arguments, restricting the antecedents and intersecting the consequents).

The $m$-ary empty system is a simple minor of the unary empty system via the scheme $H := \{h\}$, where $h \colon 1 \to m$ is given by $h(0) = 0$ (by addition of $m-1$ dummy arguments).
\end{proof}

We define the union of systems of pointed multisets componentwise, i.e., if $(\Phi_j, \Phi'_j)_{j \in J}$ is a family of pointed multisets on $A$ of a common arity $m$, then the \emph{union} of $(\Phi_j, \Phi'_j)_{j \in J}$ is
\[
\bigcup_{j \in J} (\Phi_j, \Phi'_j) := \bigl(\bigcup_{j \in J} \Phi_j, \bigcup_{j \in J} \Phi'_j \bigr).
\]

\begin{lemma}
\label{lemma:unions}
Let $(\Phi_j, \Phi'_j)_{j \in J}$ be a nonempty family of $m$-ary systems of pointed multisets on $A$. If $f \colon A^n \to A$ preserves $(\Phi_j, \Phi'_j)$ for every $j \in J$, then $f$ preserves $\bigcup_{j \in J} (\Phi_j, \Phi'_j)$.
\end{lemma}

\begin{proof}
Let $\mtrx{M} := [\mtrx{M}_1 | \mtrx{M}_2] \prec \bigcup_{j \in J} \Phi_j$. Then $\mtrx{M} \prec \Phi_j$ for some $j \in J$. By the assumption that $f \vartriangleright (\Phi_j, \Phi'_j)$, we have $[f \mtrx{M}_1 | \mtrx{M}_2] \prec \Phi'_j$, and hence $[f \mtrx{M}_1 | \mtrx{M}_2] \prec \bigcup_{j \in J} \Phi'_j$.
\end{proof}

Let $\Phi \subseteq \mathcal{M}(A^m)$, $\Phi' \subseteq A^m \times \mathcal{M}(A^m)$, $S \in \mathcal{M}(A^m)$. The \emph{quotient} of the set $\Phi$ of multisets by the multiset $S$ is defined as
\[
\Phi / S := \{S' \in \mathcal{M}(A^m) \mid S \uplus S' \in \Phi\}.
\]
The \emph{quotient} of the set $\Phi'$ of pointed multisets by the multiset $S$ is defined as
\[
\Phi' / S := \{(x, S') \in A^m \times \mathcal{M}(A^m) \mid (x, S \uplus S') \in \Phi'\}.
\]
The \emph{quotient} of the pair $(\Phi, \Phi')$ by $S$ is defined componentwise, i.e.,
$(\Phi, \Phi') / S \linebreak[0] := (\Phi / S, \Phi' / S)$.
It is easy to verify that if $(\Phi, \Phi')$ is a system of pointed multisets on $A$, then so it $(\Phi, \Phi') / S$ for every $S \in \mathcal{M}(A^m)$.

\begin{lemma}
\label{lemma:quotdistr}
Let $\Phi, \Psi \subseteq \cl{M}(A^m)$ and $S \in \mathcal{M}(A^m)$. Then
\begin{enumerate}[\rm(i)]
\item $X \in \Phi / S$ if and only if $X \uplus S \in \Phi$;
\item $(\Phi \cup \Psi) / S = (\Phi / S) \cup (\Psi / S)$.
\end{enumerate}
\end{lemma}

\begin{proof}
(i) Immediate from the definition.

(ii) By part (i) and the definition of union, we have
\begin{multline*}
X \in (\Phi \cup \Psi) / S
\Longleftrightarrow
X \uplus S \in \Phi \cup \Psi
\Longleftrightarrow
X \uplus S \in \Phi \vee X \uplus S \in \Psi \\
\Longleftrightarrow
X \in \Phi / S \vee X \in \Psi / S
\Longleftrightarrow
X \in (\Phi / S) \cup (\Psi / S).
\end{multline*}
The claimed equality thus follows.
\end{proof}

\begin{lemma}
\label{lemma:quotients}
Let $(\Phi, \Phi')$ be an $m$-ary system of pointed multisets on $A$. If $f \colon A^n \to A$ preserves $(\Phi, \Phi')$, then $f$ preserves $(\Phi, \Phi') / S$ for every multiset $S \in \mathcal{M}(A^m)$.
\end{lemma}

\begin{proof}
Let $[\mtrx{M}_1 | \mtrx{M}_2] \prec \Phi / S$. Let $\mtrx{N}$ be a matrix such that $\mtrx{N}^* = S$. Then $[\mtrx{M}_1 | \mtrx{M}_2 | \mtrx{N}] \prec \Phi$. By our assumption that $f \vartriangleright (\Phi, \Phi')$, we have $[f \mtrx{M}_1 | \mtrx{M}_2 | \mtrx{N}] \prec \Phi'$. Thus, $[f \mtrx{M}_1 | \mtrx{M}_2] \prec \Phi' / S$, and we conclude that $f \vartriangleright (\Phi, \Phi') / S$.
\end{proof}

\begin{lemma}
\label{lemma:dividends}
Assume that $(\Phi, \Phi')$ is an $m$-ary system on $A$ such that $\Omega_m^{(p)} \subseteq (\Phi, \Phi')$. If $f \colon A^n \to A$ preserves all quotients $(\Phi, \Phi') / S$ where $\card{S} \geq p$, then $f$ preserves $(\Phi, \Phi')$.
\end{lemma}

\begin{proof}
Let $[\mtrx{M}_1 | \mtrx{M}_2] \prec \Phi$, where $\mtrx{M}_1$ has $n$ columns and $\mtrx{M}_2$ has $n'$ columns. If $n' < p$, then the number of columns of $[f \mtrx{M}_1 | \mtrx{M}_2]$ is $n' + 1 \leq p$, and hence $[f \mtrx{M}_1 | \mtrx{M}_2] \prec \Phi'$. Otherwise $n' \geq p$ and, by our assumption, $f \vartriangleright (\Phi, \Phi') / \mtrx{M}_2^*$. Thus, since $\mtrx{M}_1 \prec \Phi / \mtrx{M}_2^*$, we have that $f \mtrx{M}_1 \prec \Phi' / \mtrx{M}_2^*$. Therefore $[f \mtrx{M}_1 | \mtrx{M}_2] \prec \Phi'$, and we conclude that $f \vartriangleright (\Phi, \Phi')$.
\end{proof}

For an $m$-ary system $(\Phi, \Phi') \in \mathcal{W}_A^{(m)}$ and $p \geq 0$, set $(\Phi, \Phi')^{(p)} := (\Phi^{(p)}, \linebreak[0] \Phi'^{(p)})$, where
\begin{align*}
\Phi^{(p)} & := \Phi \cap \mathcal{M}^{(p)}(A^m), \\
\Phi'^{(p)} & := \Phi' \cap (A^m \times \mathcal{M}^{(p-1)}(A^m)),
\end{align*}
that is, $(\Phi, \Phi')^{(p)} := (\Phi, \Phi') \cap \Omega_m^{(p)}$, is obtained from $(\Phi, \Phi')$ by \emph{restricting the breadth to $p$.}

\begin{lemma}
\label{lemma:breadth}
Let $(\Phi, \Phi')$ be an $m$-ary system of pointed multisets on $A$. Then $f \colon A^n \to A$ preserves $(\Phi, \Phi')$ if and only if $f$ preserves $(\Phi, \Phi')^{(p)}$ for all $p \geq 0$.
\end{lemma}

\begin{proof}
Assume first that $f \vartriangleright (\Phi, \Phi')$. Let $[\mtrx{M}_1 | \mtrx{M}_2] \prec \Phi^{(p)}$. Since $\Phi^{(p)} \subseteq \Phi$, we have that $[\mtrx{M}_1 | \mtrx{M}_2] \prec \Phi$, and hence $[f \mtrx{M}_1 | \mtrx{M}_2] \prec \Phi'$ by our assumption. The number of columns of $[f \mtrx{M}_1 | \mtrx{M}_2]$ is at most $p$, so we have that $[f \mtrx{M}_1 | \mtrx{M}_2] \prec \Phi'^{(p)}$. Thus, $f \vartriangleright (\Phi, \Phi')^{(p)}$.

Assume then that $f \vartriangleright (\Phi, \Phi')^{(p)}$ for all $p \geq 0$. Let $\mtrx{M} := [\mtrx{M}_1 | \mtrx{M}_2] \prec \Phi$, and let $q$ be the number of columns in $\mtrx{M}$. Then $[\mtrx{M}_1 | \mtrx{M}_2] \prec \Phi^{(q)}$, and hence $[f \mtrx{M}_1 | \mtrx{M}_2] \prec \Phi'^{(q)}$ by our assumption. Since $\Phi'^{(q)} \subseteq \Phi'$, we have that $[f \mtrx{M}_1 | \mtrx{M}_2] \prec \Phi'$, and we conclude that $f \vartriangleright (\Phi, \Phi')$.
\end{proof}

We say that a set $\cl{W} \subseteq \cl{W}_A$ of systems of pointed multisets is
\begin{itemize}
\item \emph{closed under quotients,} if for any $(\Phi, \Phi') \in \cl{W}$, every quotient $(\Phi, \Phi') / S$ is also in $\cl{W}$;
\item \emph{closed under dividends,} if for every system $(\Phi, \Phi') \in \cl{W}_A$, say of arity $m$, it holds that $(\Phi, \Phi') \in \cl{W}$ whenever $\Omega_m^{(p)} \subseteq (\Phi, \Phi')$ and $(\Phi, \Phi') / S \in \cl{W}$ for every multiset $S$ on $A^m$ of cardinality at least $p$;
\item \emph{locally closed,} if $(\Phi, \Phi') \in \cl{W}$ whenever $(\Phi, \Phi')^{(p)} \in \cl{W}$ for all $p \geq 0$;
\item\emph{closed under unions,} if $\bigcup_{j \in J} (\Phi_j, \Phi'_j) \in \cl{W}$ whenever $(\Phi_j, \Phi'j)_{j \in J}$ is a nonempty family of $m$-ary systems in $\cl{W}$;
\item \emph{closed under formation of conjunctive minors,} if all conjunctive minors of nonempty families of members of $\cl{W}$ are members of $\cl{W}$.
\end{itemize}

\begin{theorem}
\label{thm:spms}
Let $A$ be an arbitrary, possibly infinite nonempty set. For any set $\cl{W} \subseteq \cl{W}_A$ of systems of pointed multisets on $A$, the following two conditions are equivalent:
\begin{enumerate}[\rm(i)]
\item $\cl{W}$ is locally closed and contains the binary equality system, the unary empty system, and all unary trivial systems of breadth $p \geq 0$, and it is closed under formation of conjunctive minors, unions, quotients, and dividends.
\item $\cl{W}$ is characterized by some set $\cl{F} \subseteq \cl{O}_A$ of operations.
\end{enumerate}
\end{theorem}

In order to prove Theorem~\ref{thm:spms}, we need to extend the notions of tuple and system of pointed multisets and allow them to have infinite arities, as will be explained below. Functions remain finitary. These extended definitions have no bearing on Theorem~\ref{thm:spms} itself; they are only needed as a tool in its proof.

For any nonzero, possibly infinite ordinal $m$ (an ordinal $m$ is the set of lesser ordinals), an $m$-tuple $\vect{a} \in A^m$ is formally a map $\vect{a} \colon m \to A$. The arities of tuples and systems of pointed multisets are thus allowed to be arbitrary nonzero, possibly infinite ordinals. In minor formation schemes, the target $m$ and the members $n_j$ of the source family are also allowed to be arbitrary nonzero, possibly infinite ordinals. For systems of pointed multisets, we shall use the terms \emph{restrictive conjunctive $\infty$-minor,} \emph{extensive conjunctive $\infty$-minor,} \emph{conjunctive $\infty$-minor} and \emph{simple $\infty$-minor} to indicate a restrictive conjunctive minor, an extensive conjunctive minor, a conjunctive minor, or a simple minor via a scheme whose target and source ordinals may be infinite or finite. Thus in the sequel the use of the term ``minor'' without the prefix ``$\infty$'' continues to mean the respective minor via a scheme whose target and source ordinals are all finite. Matrices can also have infinitely many rows but only a finite number of columns; an $m \times n$ matrix $\mtrx{M} \in A^{m \times n}$, where $n$ is finite but $m$ may be finite or infinite, is an $n$-tuple of $m$-tuples $\mtrx{M} := (\vect{a}^1, \dotsc, \vect{a}^n)$ where $\vect{a}^i \colon m \to A$ for $1 \leq i \leq n$.

Let $H := (h_j)_{j \in J}$ be a minor formation scheme with target $m$, indeterminate set $V$ and source family $(n_j)_{j \in J}$, and, for each $j \in J$, let $H_j := (h^i_j)_{i \in I_j}$ be a scheme with target $n_j$, indeterminate set $V_j$ and source family $(n^i_j)_{i \in I_j}$. Assume that $V$ is disjoint from the $V_j$'s, and for distinct $j$'s the $V_j$'s are also pairwise disjoint. Then the \emph{composite scheme} $H(H_j \mid j \in J)$ is the scheme $K := (k^i_j)_{j \in J,\, i \in I_j}$ defined as follows:
\begin{enumerate}[(i)]
\item the target of $K$ is the target $m$ of $H$,
\item the source family of $K$ is $(n^i_j)_{j \in J,\, i \in I_j}$,
\item the indeterminate set of $K$ is $U := V \cup (\bigcup_{j \in J} V_j)$,
\item $k^i_j \colon n^i_j \to m \cup U$ is defined by $k^i_j := (h_j + \iota_{U V_j}) h^i_j$, where $\iota_{U V_j}$ is the canonical injection (inclusion map) from $V_j$ to $U$.
\end{enumerate}

For a set $\cl{W}$ of systems of pointed multisets on $A$ of arbitrary, possibly infinite arities, we denote by $\cl{W}^\infty$ the set of those systems which are conjunctive $\infty$-minors of families of members of $\cl{W}$. This set $\cl{W}^\infty$ is the smallest set of systems of pointed multisets containing $\cl{W}$ which is closed under formation of conjunctive $\infty$-minors, and it is called the \emph{conjunctive $\infty$-minor closure} of $\cl{W}$. Considering the formation of repeated conjunctive $\infty$-minors, we can show that the following lemma and corollary hold; these are analogues of Couceiro and Foldes's Claim~1 and Fact~1 in the proof of Theorem~3.2 in~\cite{CF}.

\begin{lemma}
If $(\Phi, \Phi')$ is a conjunctive $\infty$-minor of a nonempty family $(\Phi_j, \Phi'_j)_{j \in J}$ of systems of pointed multisets on $A$ via the scheme $H$, and, for each $j \in J$, $(\Phi_j, \Phi'_j)$ is a conjunctive $\infty$-minor of a nonempty family $(\Phi_{ji}, \Phi'_{ji})_{i \in I_j}$ via the scheme $H_j$, then $(\Phi, \Phi')$ is a conjunctive $\infty$-minor of the nonempty family $(\Phi_{ji}, \Phi'_{ji})_{j \in J,\, i \in I_j}$ via the composite scheme $K := H(H_j \mid j \in J)$.
\end{lemma}

\begin{corollary}
\label{cor:spmsmc}
Let $\cl{W} \subseteq \cl{W}_A$ be a set of finitary systems of pointed multisets, and let $\cl{W}^\infty$ be its conjunctive $\infty$-minor closure. If $\cl{W}$ is closed under formation of conjunctive minors, then $\cl{W}$ is the set of all finitary systems belonging to $\cl{W}^\infty$.
\end{corollary}

We are now ready to prove the key result needed in the proof of Theorem~\ref{thm:spms}.

\begin{lemma}
\label{lem:Kinfty}
Let $A$ be an arbitrary, possibly infinite nonempty set. Let $\cl{W} \subseteq \cl{W}_A$ be a locally closed set of finitary systems of pointed multisets that contains the binary equality system, the unary empty system, and all unary trivial systems of breadth $p \geq 0$, and is closed under formation of conjunctive minors, unions, quotients, and dividends. Let $\cl{W}^\infty$ be the conjunctive $\infty$-minor closure of $\cl{W}$. Let $(\Phi, \Phi') \in \cl{W}_A \setminus \cl{W}$ be finitary. Then there exists a function in $\cl{O}_A$ which preserves every system in $\cl{W}^\infty$ but does not preserve $(\Phi, \Phi')$.
\end{lemma}

\begin{proof}
We shall construct a function $g$ that preserves all systems in $\cl{W}^\infty$ but does not preserve $(\Phi, \Phi')$.

Note that, by Corollary~\ref{cor:spmsmc}, $(\Phi, \Phi')$ cannot be in $\cl{W}^\infty$. Let $m$ be the arity of $(\Phi, \Phi')$. Since $\cl{W}$ is locally closed and $(\Phi, \Phi') \notin \cl{W}$, there is an integer $p$ such that $(\Phi, \Phi')^{(p)} := (\Phi, \Phi') \cap \Omega_m^{(p)} \notin \cl{W}$; let $n$ be the smallest such integer. By Lemma~\ref{lemma:breadth}, each function not preserving $(\Phi, \Phi')^{(n)}$ does not preserve $(\Phi, \Phi')$ either, so we can consider $(\Phi, \Phi')^{(n)}$ instead of $(\Phi, \Phi')$. Due to the minimality of $n$, the breadth of $(\Phi, \Phi')^{(n)}$ is $n$. Observe that $(\Phi, \Phi')$ is not the trivial system of breadth $n$ nor the empty system, because these are members of $\cl{W}$. Thus, $n \geq 1$.

We can assume that $(\Phi, \Phi')$ is a minimal nonmember of $\cl{W}$ with respect to identification of rows, i.e., every simple minor of $(\Phi, \Phi')$ obtained by identifying some rows of $(\Phi, \Phi')$ is a member of $\cl{W}$. If this is not the case, then we can identify some rows of $(\Phi, \Phi')$ to obtain a minimal nonmember $(\tilde\Phi, \tilde\Phi')$ of $\cl{W}$ and consider the cluster $(\tilde\Phi, \tilde\Phi')$ instead of $(\Phi, \Phi')$. Note that by Lemma~\ref{lemma:cm}, each function not preserving $(\tilde\Phi, \tilde\Phi')$ does not preserve $(\Phi, \Phi')$ either.

We can also assume that $(\Phi, \Phi')$ is a minimal nonmember of $\cl{W}$ with respect to taking quotients, i.e., whenever $S \neq \varepsilon$, we have that $(\Phi, \Phi') / S \in \cl{K}$. If this is not the case, then consider a minimal nonmember $(\Phi, \Phi') / S$ of $\cl{W}$ instead of $(\Phi, \Phi')$. By Lemma~\ref{lemma:quotients}, each function not preserving $(\Phi, \Phi') / S$ does not preserve $(\Phi, \Phi')$ either.

The fact that $(\Phi, \Phi')$ is a minimal nonmember of $\cl{W}$ with respect to taking quotients implies that $\Omega_m^{(1)} \not\subseteq (\Phi, \Phi')$. For, suppose, on the contrary, that $\Omega_m^{(1)} \subseteq (\Phi, \Phi')$. Since all quotients $(\Phi, \Phi') / S$ where $\card{S} \geq 1$ are in $\cl{W}$ and $\cl{W}$ is closed under dividends, we have that $(\Phi, \Phi') \in \cl{W}$, a contradiction.

Let $(\Psi, \Psi') := \bigcup \{(P, P') \in \cl{W} \mid (P, P') \subseteq (\Phi, \Phi')\}$, i.e., $(\Psi, \Psi')$ is the largest system in $\cl{W}$ such that $(\Psi, \Psi') \subseteq (\Phi, \Phi')$. Note that this is not the empty union, because the empty system is a member of $\cl{W}$. It is clear that $(\Psi, \Psi') \neq (\Phi, \Phi')$. Furthermore, $\Psi \subsetneq \Phi$, for if it were the case that $\Psi = \Phi$, then $(\Phi, \Phi')$ would be a conjunctive minor of $(\Psi, \Psi')$ by extending the consequent and hence $(\Phi, \Phi')$ would be a member of $\cl{W}$, a contradiction. Since $n$ was chosen to be the smallest integer satisfying $(\Phi, \Phi')^{(n)} \notin \cl{W}$, we have that $(\Phi, \Phi')^{(n-1)} \in \cl{W}$ and since $(\Phi, \Phi')^{(n-1)} \subseteq (\Phi, \Phi')^{(n)}$, it holds that $(\Phi, \Phi')^{(n-1)} \subseteq (\Psi, \Psi')$. Thus there is a multiset $Q \in \Phi \setminus \Psi$ with $\card{Q} = n$. Let $\mtrx{D} := (\vect{d}^1, \dotsc, \vect{d}^n)$ be an $m \times n$ matrix whose multiset of columns equals $Q$.

The rows of $\mtrx{D}$ are pairwise distinct. Suppose, for the sake of contradiction, that rows $i$ and $j$ of $\mtrx{D}$ coincide. Since $(\Phi, \Phi')$ is a minimal nonmember of $\cl{W}$ with respect to identification of rows, by identifying rows $i$ and $j$ of $(\Phi, \Phi')$, we obtain a system $(\tilde\Phi, \tilde\Phi')$ that is in $\cl{W}$. By adding a dummy row in the place of the row that got deleted when we identified rows $i$ and $j$, and finally by intersecting with the conjunctive minor of the binary equality cluster whose rows $i$ and $j$ are equal (the overall effect of the operations performed above is the selection of exactly those multisets in $\Phi$ and pointed multisets in $\Phi'$ whose rows $i$ and $j$ coincide), we obtain a system $(\bar\Phi, \bar\Phi') \in \cl{W}$ such that $Q \in \bar\Phi$ and $(\bar\Phi, \bar\Phi') \subseteq (\Phi, \Phi')$. But this is impossible by the choice of $Q$.

Let $(\hat\Phi, \hat\Phi') := (\Phi, \Phi') \cup \Omega_m^{(1)}$. We claim that for $S \neq \varepsilon$, $(\hat\Phi, \hat\Phi') / S = (\Phi, \Phi') / S$ or $(\hat\Phi, \hat\Phi') / S = (\Phi, \Phi') / S \cup \Omega_m^{(0)}$. For, by Lemma~\ref{lemma:quotdistr} we have
\[
(\hat\Phi, \hat\Phi') / S = ((\Phi, \Phi') \cup \Omega_m^{(1)}) / S = (\Phi, \Phi') / S \cup \Omega_m^{(1)} / S.
\]
If $\card{S} > 1$, then $\Omega_m^{(1)} / S = (\emptyset, \emptyset)$; hence $(\hat\Phi, \hat\Phi') = (\Phi, \Phi')$. If $\card{S} = 1$, then $\Omega_m^{(1)} / S = (\{\varepsilon\}, \emptyset) = \Omega_m^{(0)}$; hence $(\hat\Phi, \hat\Phi') / S = (\Phi, \Phi') / S \cup \Omega_m^{(0)}$.

Since $(\Phi, \Phi')$ is a minimal nonmember of $\cl{W}$ with respect to quotients, $\Omega_m^{(0)} \in \cl{W}$ and $\cl{W}$ is closed under unions, by the above claim we have that $(\hat\Phi, \hat\Phi') / S \in \cl{W}$ whenever $\card{S} \geq 1$. Since $\cl{W}$ is closed under dividends, we have that $(\hat\Phi, \hat\Phi') \in \cl{W}$.

Let $(\Upsilon, \Upsilon') := \bigcap \{(P, P') \in \cl{W} \mid Q \in P\}$, i.e., $(\Upsilon, \Upsilon')$ is the smallest system in $\cl{W}$ such that $Q \in \Upsilon$. Note that this is not the empty intersection, because $(\hat\Phi, \hat\Phi')$ is a member of $\cl{W}$, as shown above, and $Q \in \hat\Phi$; thus $(\Upsilon, \Upsilon') \subseteq (\hat\Phi, \hat\Phi')$.

We claim that $\Upsilon' \not\subseteq \Phi'$. Suppose, on the contrary, that $\Upsilon' \subseteq \Phi'$. Then we must have that $\Phi \not\subseteq \Upsilon$. For, if it were the case that $\Phi \subseteq \Upsilon$, then $(\Phi, \Phi')$ would be a conjuctive minor of $(\Upsilon, \Upsilon')$ (by restricting the antecedent and extending the consequent) and hence $(\Phi, \Phi')$ would be a member of $\cl{W}$, a contradiction. Consider
\[
(\Lambda, \Lambda') := (\Psi, \Psi') \cup (\Upsilon \cap \Phi, \Upsilon').
\]
Let us first verify that the pair $(\Upsilon \cap \Phi, \Upsilon')$ is actually a system of pointed multisets. Since $(\Upsilon, \Upsilon') \in \cl{W}_A$, it holds that $\Upsilon'$ is downward closed and $\{x\} \uplus S \in \Upsilon$ for every $(x, S) \in \Upsilon'$. By the assumption that $\Upsilon' \subseteq \Phi'$ and by the fact that $(\Phi, \Phi') \in \cl{W}_A$, it also holds that $\{x\} \uplus S \in \Phi$ for every $(x, S) \in \Upsilon'$. Indeed, $(\Upsilon \cap \Phi, \Upsilon') \in \cl{W}_A$ as claimed.

The system $(\Upsilon \cap \Phi, \Upsilon')$ is a conjunctive minor of $(\Upsilon, \Upsilon') \in \cl{W}$ (by restricting the antecedent) and hence it is a member of $\cl{W}$. Since $(\Psi, \Psi')$ is also a member of $\cl{W}$ and $\cl{W}$ is closed under unions, we have that $(\Lambda, \Lambda') \in \cl{W}$. We have that
\[
(\Psi, \Psi') \subsetneq (\Lambda, \Lambda') \subsetneq (\Phi, \Phi'),
\]
where the first inclusion clearly holds by the definition of $(\Lambda, \Lambda')$, and the inclusion is strict, because $Q \in \Upsilon \cap \Phi$ but $Q \notin \Psi$. The second inclusion holds, because $(\Psi, \Psi') \subseteq (\Phi, \Phi')$ by the definition of $(\Psi, \Psi')$, $\Upsilon \cap \Phi \subseteq \Phi$ by the definition of intersection and $\Upsilon' \subseteq \Phi'$ by our assumption. The second inclusion is strict, because $(\Lambda, \Lambda') \in \cl{W}$ but $(\Phi, \Phi') \notin \cl{W}$. We have reached a contradiction, because $(\Psi, \Psi')$ is by definition the largest member of $\cl{W}$ that is componentwise included in $(\Phi, \Phi')$. This completes the proof of the claim that $\Upsilon' \not\subseteq \Phi'$.

We have shown above that $\Upsilon' \subseteq \hat\Phi'$ but $\Upsilon' \not\subseteq \Phi'$. We conclude that there exists an $m$-tuple $\vect{s} \in A^m$ such that $(\vect{s}, \varepsilon) \in \Upsilon' \setminus \Phi'$.

Let $\mtrx{M} := (\vect{m}^1, \dotsc, \vect{m}^n)$ be a $\mu \times n$ matrix whose first $m$ rows are the rows of $\mtrx{D}$ (i.e., $\bigl( \vect{m}^1(i), \dotsc, \vect{m}^n(i) \bigr) = \bigl( \vect{d}^1(i), \dotsc, \vect{d}^n(i) \bigr)$ for every $i \in m$) and whose other rows are the remaining distinct $n$-tuples in $A^n$; every $n$-tuple in $A^n$ is a row of $\mtrx{M}$ and there is no repetition of rows in $\mtrx{M}$. Note that $m \leq \mu$ and $\mu$ is infinite if and only if $A$ is infinite.

Let $(\Theta, \Theta') := \bigcap \{(P, P') \in \cl{W}^\infty \mid \mtrx{M} \prec P\}$. There must exist a $\mu$-tuple $\vect{u} := (u_t \mid t \in \mu)$ in $A^\mu$ such that $\vect{u}(i) = \vect{s}(i)$ for all $i \in m$ and $(\vect{u}, \varepsilon) \in \Theta'$. For, suppose that this is not the case. Let $(\tilde\Theta, \tilde\Theta')$ be the projection of $(\Theta, \Theta')$ to its first $m$ coordinates. Then $(\tilde\Theta, \tilde\Theta') \in \cl{W}$ and $Q \in \tilde\Theta$ but $(\vect{s}, \varepsilon) \notin \tilde\Theta'$. This contradicts the choice of $\vect{s}$.

We can now define a function $g \colon A^n \to A$ by the rule $g \mtrx{M} = \vect{u}$. The definition is valid, because every $n$-tuple in $A^n$ occurs exactly once as a row of $\mtrx{M}$. It is clear that $g \not\vartriangleright (\Phi, \Phi')$, because $\mtrx{D} \prec \Phi$ but $g \mtrx{D} = \vect{s} \nprec \Phi'$.

We need to show that every system in $\cl{W}^\infty$ is preserved by $g$. Suppose, on the contrary, that there is a $\rho$-ary system $(\Phi_0, \Phi'_0) \in \cl{W}^\infty$, possibly infinitary, which is not preserved by $g$. Thus, for some $\rho \times n'$ matrix $\mtrx{N} := (\vect{c}^1, \dotsc, \vect{c}^{n'}) \prec \Phi_0$, with $\mtrx{N}_0 := (\vect{c}^1, \dotsc, \vect{c}^n)$, $\mtrx{N}_1 := (\vect{c}^{n+1}, \dotsc, \vect{c}^{n'})$, we have $[g \mtrx{N}_0 | \mtrx{N}_1] \nprec \Phi'_0$. Let $(\Phi_1, \Phi'_1) := (\Phi_0, \Phi'_0) / \mtrx{N}_1^*$. Since $\cl{W}$ is closed under quotients, $(\Phi_1, \Phi'_1) \in \cl{W}$. We have that $\mtrx{N}_0 \prec \Phi_1$ but $g \mtrx{N}_0 \nprec \Phi'_1$, so $g$ does not preserve $(\Phi_1, \Phi'_1)$ either. Define $h \colon \rho \to \mu$ to be any map such that
\[
\bigl( \vect{c}^1(i), \dotsc, \vect{c}^n(i) \bigr) = \bigl( (\vect{m}^1 h)(i), \dotsc, (\vect{m}^n h)(i) \bigr)
\]
for every $i \in \rho$, i.e., row $i$ of $\mtrx{N}_0$ is the same as row $h(i)$ of $\mtrx{M}$, for each $i \in \rho$. Let $(\Phi_h, \Phi'_h)$ be the $\mu$-ary simple $\infty$-minor of $(\Phi_1, \Phi'_1)$ via $H := \{h\}$. Note that $(\Phi_h, \Phi'_h) \in \cl{W}^\infty$.

We claim that $\mtrx{M} \prec \Phi_h$. To prove this, by the definition of simple $\infty$-minor, it is enough to show that $(\vect{m}^1 h, \dotsc, \vect{m}^n h) \linebreak[0] \prec \Phi_1$. In fact, we have for $1 \leq j \leq n$,
\[
\vect{m}^j h = (\vect{m}^j h(i) \mid i \in \rho) = (\vect{c}^j (i) \mid i \in \rho) = \vect{c}^j,
\]
and $(\vect{c}^1, \dotsc, \vect{c}^n) = \mtrx{N}_0 \prec \Phi_1$.

Next we claim that $(\vect{u}, \varepsilon) \notin \Phi'_h$. For this, by the definition of simple $\infty$-minor, it is enough to show that $\vect{u} h \nprec \Phi'_1$. For every $i \in \rho$, we have
\[
\begin{split}
(\vect{u} h)(i)
&= \bigl( g(\vect{m}^1, \dotsc, \vect{m}^n) h \bigr) (i) \\
&= g \bigl( (\vect{m}^1 h)(i), \dotsc, (\vect{m}^n h)(i) \bigr)
= g \bigl( \vect{c}^1(i), \dotsc, \vect{c}^n(i) \bigr).
\end{split}
\]
Thus $\vect{u} h = g \mtrx{N}_0$. Since $g \mtrx{N}_0 \nprec \Phi'_1$, we conclude that $(\vect{u}, \varepsilon) \notin \Phi'_h$.

Thus, $(\Phi_h, \Phi'_h) \in \cl{W}^\infty$, $\mtrx{M} \prec \Phi_h$ but $(\vect{u}, \varepsilon) \notin \Phi'_h$. By the choice of $\vect{u}$, this is impossible, and we have reached a contradiction.
\end{proof}

\begin{proof}[Proof of Theorem~\ref{thm:spms}]
$\text{(ii)} \implies \text{(i)}$: It is clear that every function preserves the equality, empty, and trivial systems. By Lemmas~\ref{lemma:cm},~\ref{lemma:unions},~\ref{lemma:quotients}, and~\ref{lemma:dividends}, $\cl{W}$ is closed under formation of conjunctive minors, unions, quotients, and dividends.

It remains to show that $\cl{W}$ is locally closed. Suppose on the contrary that there is a system $(\Phi, \Phi') \in \cl{W}_A \setminus \cl{W}$, say of arity $m$, such that $(\Phi, \Phi')^{(p)} = (\Phi, \Phi') \cap \Omega_m^{(p)} \in \cl{W}$ for all $p \geq 0$. By (ii), there is an operation $f \colon A^n \to A$ that preserves every system in $\cl{W}$ but does not preserve $(\Phi, \Phi')$. Thus, there is a $p \geq 0$ and an $m \times p$ matrix $\mtrx{M} := [\mtrx{M}_1 | \mtrx{M}_2] \prec \Phi$ such that $[f \mtrx{M}_1 | \mtrx{M}_2] \nprec \Phi'$. By our assumption, $(\Phi, \Phi')^{(p)} \in \cl{W}$, but we have that $[\mtrx{M}_1 | \mtrx{M}_2] \prec \Phi^{(p)}$ and $[f \mtrx{M}_1 | \mtrx{M}_2] \nprec \Phi'^{(p)}$, which is a contradiction to the fact that $f \vartriangleright (\Phi, \Phi')^{(p)}$.

\smallskip
$\text{(i)} \implies \text{(ii)}$:
By Lemma~\ref{lem:Kinfty}, for every system $(\Phi, \Phi') \in \cl{W}_A \setminus \cl{W}$, there is a function in $\cl{O}_A$ which preserves every system in $\cl{W}$ but does not preserve $(\Phi, \Phi')$. The set of these ``separating'' functions, for each $(\Phi, \Phi') \in \cl{W}_A \setminus \cl{W}$, characterizes $\cl{W}$.
\end{proof}

%%%%%%%%%%%%%%%%%%%%%%%%%%%%%%%%%%%%%%%%%%%%%%%

\section{The number of closed sets}

In this section, we will show that the closure system of the subalgebras of $(\mathcal{O}_A; \zeta, \tau, \nabla, \ast)$ is uncountable whenever $\card{A} \geq 2$. We will first recall basic notions related to terms and term operations, following the notation and terminology presented in~\cite{DW}.

For a natural number $n \geq 1$, let $X_n := \{x_1, \dots, x_n\}$ be a set of \emph{variables.} Let $\{f_i \mid i \in I\}$ be a set of \emph{operation symbols,} disjoint from the variables, and assign to each operation symbol $f_i$ a natural number $n_i$, called the \emph{arity} of $f_i$. The sequence $\tau := (n_i)_{i \in I}$ is called a \emph{type.} The \emph{$n$-ary terms} of type $\tau$ are defined in the following inductive way:
\begin{enumerate}[\indent (i)]
\item Every variable $x_i \in X_n$ is an $n$-ary term.

\item If $f_i$ is an $n_i$-ary operation symbol and $t_1, \dots, t_{n_i}$ are $n$-ary terms, then $f_i(t_1, \dots, t_{n_i})$ is an $n$-ary term.

\item The set $W_\tau(X_n)$ of all $n$-ary terms is the smallest set which contains the variables $x_1, \dots, x_n$ and which is closed under the finite application of (ii).
\end{enumerate}

Every $n$-ary term is also an $m$-ary term for every $m \geq n$.
Let $X := \bigcup_{n \geq 1} X_n = \{x_1, x_2, \dots\}$. We denote by $W_\tau(X)$ the set of all terms of type $\tau$ over the countably infinite alphabet $X$:
\[
W_\tau(X) := \bigcup_{n \geq 1} W_\tau(X_n).
\] 

A term is \emph{linear,} if it contains no multiple occurrences of the same variable.
We denote by $W_\tau^\mathrm{lin}(X_n)$ the set of all $n$-ary linear terms of type $\tau$ over the alphabet $X_n$, and we denote by $W_\tau^\mathrm{lin}(X)$ the set of all linear terms of type $\tau$ over $X$.

The number of occurrences of operation symbols in a term is called the \emph{complexity} of the term. 
A term $s$ is a \emph{subterm} of a term $t$ if $t = usv$ for some words $u$ and $v$.
The subterms of a linear term are linear.

Let $\mathcal{A} = (A; (f_i)_{i \in I})$ be an algebra of type $\tau$, i.e., each fundamental operation $f_i$ has arity $n_i$, and let $t$ be an $n$-ary term of type $\tau$ over $X$. The term $t$ induces an $n$-ary operation $t^\mathcal{A}$ on $A$ (see Definition~5.2.1 in~\cite{DW}). We call an operation induced by a linear term a \emph{linear term operation.} The set of all $n$-ary linear term operations of the algebra $\mathcal{A}$ is denoted by $W_\tau^\mathrm{lin}(X_n)^\mathcal{A}$, and the set of all finitary linear term operations of the algebra $\mathcal{A}$ is denoted by $W_\tau^\mathrm{lin}(X)^\mathcal{A}$.

For a set $F \subseteq \cl{O}_A$, the universe of the subalgebra of $(\cl{O}_A; \zeta, \tau, \nabla, \ast)$ generated by $F$ is denoted by $\langle F \rangle$. The following theorem shows that linear terms are related to subuniverses of $(\cl{O}_A; \zeta, \tau, \nabla, \ast)$ much in the same way as terms are related to clones.

\begin{theorem}
\label{thm:linear}
Let $\mathcal{A} = (A; (f_i^A)_{i \in I})$ be an algebra of type $\tau$, and let $W_\tau^\mathrm{lin}(X)$ be the set of all linear terms of type $\tau$ over $X$. Then $W_\tau^\mathrm{lin}(X)^{\mathcal{A}}$ is a subuniverse of $(\cl{O}_A; \zeta, \tau, \nabla, \ast)$ that contains all projections on $A$. Moreover, $W_\tau^\mathrm{lin}(X)^\mathcal{A} = \langle \{f_i \mid i \in I\} \cup \cl{E}_A \rangle$.
\end{theorem}

\begin{proof}
Since for all $1 \leq i \leq n$, $x_i \in X_n$, we have $x_i^\mathcal{A} = e_i^{n,A} \in W_\tau^\mathrm{lin}(X)^\mathcal{A}$. Thus $W_\tau^\mathrm{lin}(X)^\mathcal{A}$ contains all projections on $A$. Let $f, g \in W_\tau^\mathrm{lin}(X)^\mathcal{A}$, say $f$ is $n$-ary, $g$ is $m$-ary. Then there exist linear terms $t \in W_\tau^\mathrm{lin}(X_n)$, $s \in W_\tau^\mathrm{lin}(X_m)$ such that $f^\mathcal{A} = f$, $s^\mathcal{A} = g$. Then
\begin{itemize}
\item $t(x_2, x_3, \dots, x_n, x_1) \in W_\tau^\mathrm{lin}(X_n)$ and $t(x_2, x_3, \dots, x_n, x_1)^\mathcal{A} = \zeta f$,
\item $t(x_2, x_1, x_3, \dots, x_n) \in W_\tau^\mathrm{lin}(X_n)$ and $t(x_2, x_1, x_3, \dots, x_n)^\mathcal{A} = \tau f$,
\item $t(x_2, \dots, x_{n+1}) \in W_\tau^\mathrm{lin}(X_{n+1})$ and $t(x_2, \dots, x_{n+1})^\mathcal{A} = \nabla f$,
\item $t(s, x_{m+1}, x_{m+2}, \dots, x_{m+n-1}) \in W_\tau^\mathrm{lin}(X_{m+n-1})$ and \newline $t(s, x_{m+1}, x_{m+2}, \dots, x_{m+n-1})^\mathcal{A} = f \ast g$.
\end{itemize}
Thus, $\zeta f, \tau f, \nabla f, f \ast g \in W_\tau^\mathrm{lin}(X)^{\mathcal{A}}$. Therefore, $W_\tau^\mathrm{lin}(X)^{\mathcal{A}}$ is a subuniverse of $(\cl{O}_A; \zeta, \tau, \nabla, \ast)$.

It is clear that $\{f_i^\mathcal{A} \mid i \in I\} \cup \cl{E}_A \subseteq W_\tau^\mathrm{lin}(X)^\mathcal{A}$, and so $\langle \{f_i^\mathcal{A} \mid i \in I\} \cup \cl{E}_A \rangle \subseteq W_\tau^\mathrm{lin}(X)^\mathcal{A}$. We will show the converse inclusion by induction on the complexity of a term $t$.
If $t = x_i \in X_n$, then $t^\mathcal{A} = e_i^{n,A} \in \langle \{f_i^\mathcal{A} \mid i \in I\} \cup \cl{E}_A \rangle$.
Otherwise $t = f_i(t_1, \dots, t_{n_i})$ is a linear term and $t^\mathcal{A} \in W_\tau^\mathrm{lin}(X)^\mathcal{A}$.
Then there exist numbers $m_1, \dots, m_{n_i} \geq 1$ and an injective map
\[
\sigma \colon \{(j,k) \in \nat \times \nat \mid 1 \leq j \leq n_i, \, 1 \leq k \leq m_j\} \to \{1, \dots, n\}
\]
such that the variables occurring in the linear term $t_j$ ($1 \leq j \leq n_i$) are precisely $x_{\sigma(j,1)}, \dots, x_{\sigma(j,m_j)}$. For $1 \leq j \leq n_i$, let $u_j$ be the $m_j$-ary term that is obtained by replacing the occurrence of $x_{\sigma(j,\ell)}$ by $x_\ell$ for each $1 \leq \ell \leq m_j$. Note that then we clearly have that
\[
t_j^\mathcal{A} = u_j^\mathcal{A}(e_{\sigma(j,1)}^{n,A}, \dots, e_{\sigma(j,m_j)}^{n,A}) = u_j(x_{\sigma(j,1)}, \dots, x_{\sigma(j,m_j)})^\mathcal{A}.
\]
It is clear that $f_i^\mathcal{A} \in \langle \{f_i^\mathcal{A} \mid i \in I\} \cup \cl{E}_A \rangle$, and by our induction hypothesis it also holds that $u_1^\mathcal{A}, \dots, u_{n_i}^\mathcal{A} \in \langle \{f_i^\mathcal{A} \mid i \in I\} \cup \cl{E}_A \rangle$. Then repeated applications of $\zeta$ and $\ast$ show that the functions
\begin{align*}
& \zeta f_i^\mathcal{A}, \\
& \zeta f_i^\mathcal{A} \ast u_{n_i}^\mathcal{A}, \\
\zeta( & \zeta f_i^\mathcal{A} \ast u_{n_i}^\mathcal{A}), \\
\zeta( & \zeta f_i^\mathcal{A} \ast u_{n_i}^\mathcal{A}) \ast u_{n_i - 1}^\mathcal{A}, \\
\zeta( \zeta( & \zeta f_i^\mathcal{A} \ast u_{n_i}^\mathcal{A}) \ast u_{n_i - 1}^\mathcal{A}), \\
\zeta( \zeta( & \zeta f_i^\mathcal{A} \ast u_{n_i}^\mathcal{A}) \ast u_{n_i - 1}^\mathcal{A}) \ast u_{n_i - 2}^\mathcal{A}, \\
& \qquad \vdots \\
\zeta( \zeta( \cdots ( \zeta( & \zeta f_i^\mathcal{A} \ast u_{n_i}^\mathcal{A}) \ast u_{n_i - 1}^\mathcal{A}) \ast \cdots ) \ast u_2^\mathcal{A} ) \ast u_1^\mathcal{A}
\end{align*}
are all in $\langle \{f_i^\mathcal{A} \mid i \in I\} \cup \cl{E}_A \rangle$. Note that
\begin{align*}
& \zeta f_i^\mathcal{A}(a_1, \dots, a_{n_i}) =
    f_i^\mathcal{A}(a_2, \dots, a_{n_i}, a_1), \\
& \zeta f_i^\mathcal{A} \ast u_{n_i}^\mathcal{A}(a_1, \dots, a_{m_{n_i} + n_i - 1}) = \\
&   \qquad f_i^\mathcal{A}(a_{m_{n_i} + 1}, \dots, a_{m_{n_i} + n_i - 1}, u_{n_i}^\mathcal{A}(a_1, \dots, a_{m_{n_i}})), \\
& \zeta( \zeta f_i^\mathcal{A} \ast u_{n_i}^\mathcal{A})(a_1, \dots, a_{m_{n_i} + n_i - 1}) = \\
&   \qquad f_i^\mathcal{A}(a_{m_{n_i} + 2}, \dots, a_{m_{n_i} + n_i - 1}, a_1, u_{n_i}^\mathcal{A}(a_2, \dots, a_{m_{n_i} + 1})), \\
& \zeta( \zeta f_i^\mathcal{A} \ast u_{n_i}^\mathcal{A}) \ast u_{n_i - 1}^\mathcal{A}(a_1, \dots, a_{m_{n_i - 1} + m_{n_i} + n_i - 2}) = \\
&   \qquad f_i^\mathcal{A}(a_{m_{n_i - 1} + m_{n_i} + 1}, \dots, a_{m_{n_i - 1} + m_{n_i} + n_i - 2}, \\ & \qquad\qquad u_{n_i - 1}^\mathcal{A}(a_1, \dots, a_{m_{n_i - 1}}), u_{n_i}^\mathcal{A}(a_{m_{n_i - 1} + 1}, \dots, a_{m_{n_i - 1} + m_{n_i}})), \\
& \qquad\qquad\qquad\qquad\qquad \vdots \\
& \zeta( \zeta( \cdots ( \zeta( \zeta f_i^\mathcal{A} \ast u_{n_i}^\mathcal{A}) \ast u_{n_i - 1}^\mathcal{A}) \ast \cdots ) \ast u_2^\mathcal{A} ) \ast u_1^\mathcal{A} (a_1, \dots, a_{m_1 + m_2 + \dots + m_{n_i}}) = \\
&   \qquad f_i^\mathcal{A}(u_1^\mathcal{A}(a_1, \dots, a_{m_1}), u_2^\mathcal{A}(a_{m_1 + 1}, \dots, a_{m_1 + m_2}), \dots, \\ & \qquad\qquad u_{n_i}^\mathcal{A}(a_{m_1 + m_2 + \dots + m_{n_i - 1} + 1}, \dots, a_{m_1 + m_2 + \dots + m_{n_i}})).
\end{align*}
Furthermore, repeated applications of $\zeta$, $\tau$ and $\nabla$ yield that the $n$-ary function $g$ given by
\[
g(a_1, \dots, a_n) =
f_i^\mathcal{A}(u_1^\mathcal{A}(a_{\sigma(1,1)}, \dots, a_{\sigma(1,m_1)}),
% u_2^\mathcal{A}(a_{\sigma(2,1)}, \dots, a_{\sigma(2,m_2)}),
\dots, u_{n_i}^\mathcal{A}(a_{\sigma(n_i,1)}, \dots, a_{\sigma(n_i,m_{n_i})}))
\]
is in $\langle \{f_i^\mathcal{A} \mid i \in I\} \cup \cl{E}_A \rangle$. We clearly have that $t^\mathcal{A} = g$. Therefore $\{f_i^\mathcal{A} \mid i \in I\} \cup \cl{E}_A \supseteq W_\tau^\mathrm{lin}(X)^\mathcal{A}$, and the claimed equality holds.
\end{proof}

Assume that $0$ and $1$ are distinct elements of $A$. For each integer $n \geq 3$, define the function $\mu_n \colon A^n \to A$ by
\[
\mu_n(a_1, \dots, a_n) =
\begin{cases}
1 & \text{if $(a_1, \dots, a_n) \in \{0,1\}^n$ and} \\
  & \qquad \text{$\lvert \{i \in \{1, \dots, n\} : a_i = 1\} \rvert \in \{1, n - 1\}$,} \\
0 & \text{otherwise.}
\end{cases}
\]
In the particular case that $A = \{0,1\}$, the $\mu_n$ are the Boolean functions defined by Pippenger~\cite[Proposition~3.4]{Pippenger}. Observe that $\mu_n(0, \dots, 0) = 0$ for every $n \geq 3$.

\begin{lemma}
\label{lem:mus}
Let $I \subseteq \nat \setminus \{0, 1, 2\}$ and $k \in \nat \setminus \{0, 1, 2\}$. Then $\mu_k \in \langle \{\mu_i \mid i \in I\} \cup \cl{E}_A \rangle$ if and only if $k \in I$.
\end{lemma}

\begin{proof}
If $k \in I$, then obviously $\mu_k \in \langle \{\mu_i \mid i \in I\} \cup \cl{E}_{\{0,1\}} \rangle$. Assume then that $k \notin I$. Let $\mathcal{A} = (A; (\mu_i)_{i \in I})$. By Theorem~\ref{thm:linear}, it is enough to show that there is no linear term $t$ of type $\tau := (i)_{i \in I}$ such that $t^\mathcal{A} = \mu_k$. Thus, let $t$ be a $k$-ary linear term of type $\tau$. It is clear that a $k$-ary linear term does not contain any operation symbols of arity greater than $k$, and since $k \notin I$, the operation symbols occurring in $t$ have arity less than $k$. If $t = x_j$, then $t^\mathcal{A} = e_j^{k,A}$, but clearly $\mu_k$ is not a projection. Otherwise, $t$ has a subterm $p$ of the form $f_\ell(x_{i_1}, \dots, x_{i_\ell})$, where $\ell < k$ and $i_1, \dots, i_\ell$ are pairwise distinct. Let $\vect{x} \in \{0,1\}^k$ be a $k$-tuple with $1$'s at exactly $\ell - 1$ positions among $i_1, \dots, i_\ell$ and $0$'s at all remaining positions.

\textit{Claim.} For every subterm $s$ of $t$ that contains $p$ as a subterm, $s^\mathcal{A}(\vect{x}) = 1$.

\textit{Proof of Claim.}
We first define the \emph{depth} $d(s,t)$ of a subterm $s$ in the linear term $t$ recursively as follows:
$d(t,t) = 0$; and if $s = f_i(t_1, \dots, t_{n_i})$ is a subterm of $t$ with $d(s,t) = d$, then $d(t_j,t) = d + 1$ for $1 \leq j \leq n_i$.

We proceed by induction on the depth $d$ of the subterm $p$ in $s$. If $d = 0$ then $s = p$, and so $s^\mathcal{A}(\vect{x}) = p^\mathcal{A}(\vect{x}) = \mu_\ell(\vect{x}) = 1$. Assume that the claim holds for $d = q$ for some $q \geq 0$. Let then $d = q + 1$. Then $s = f_m(t_1, \dots, t_m)$ for some $m < k$, and there is an $n \in \{1, \dots, m\}$ such that $p$ is a subterm of $t_i$ and the depth of $p$ in $t_i$ is $d$. By the induction hypothesis, $t_n^\mathcal{A}(\vect{x}) = 1$. Furthermore, we have that for all $p \neq n$,
\[
t_p^\mathcal{A}(\vect{x}) = t_p^\mathcal{A}(0, \dots, 0) = 0,
\]
where the first equality holds because the variables $x_{i_1}, \dots, x_{i_\ell}$ do not occur in $t_p$ since $t$ is a linear term; and the second equality holds because the fundamental operations of $\mathcal{A}$ preserve $0$, and hence so do all term operations on $\mathcal{A}$. Thus,
\[
s^\mathcal{A}(\vect{x})
= f_m^\mathcal{A}(t_1^\mathcal{A}(\vect{x}), \dots, t_m^\mathcal{A}(\vect{x}))
= \mu_m(0, \dots, 0, 1, 0, \dots, 0)
= 1,
\]
as claimed.
\qquad $\diamond$

By the Claim, we have in particular that $t^\mathcal{A}(\vect{x}) = 1$. However, since $3 \leq \ell < k$, we have that $1 < \ell - 1 < k - 1$; hence $\mu_k(\vect{x}) = 0$.
\end{proof}

\begin{theorem}
\label{thm:size}
Let $\card{A} \geq 2$.
\begin{enumerate}[\rm (i)]
\item The set of subalgebras of $(\cl{O}_A; \zeta, \tau, \nabla, \ast)$ containing all projections is uncountable.
\item The set of subalgebras of $(\cl{O}_A; \zeta, \tau, \nabla, \ast)$ is uncountable.
\end{enumerate}
\end{theorem}

\begin{proof}
(i) By Lemma~\ref{lem:mus}, if $I, J \subseteq \nat \setminus \{0, 1, 2\}$ and $I \neq J$, then $\langle \{\mu_i \mid i \in I\} \cup \cl{E}_A \rangle \neq \langle \{\mu_i \mid i \in J\} \cup \cl{E}_A \rangle$. Thus, there are uncountably many subalgebras of $(\cl{O}_A; \zeta, \tau, \nabla, \ast)$ containing all projections.

(ii) An immediate consequence of (i).
\end{proof}

%%%%%%%%%%%%%%%%%%%%%%
\begin{table}
\begin{center}
\begin{tabular}{|c|p{2.75cm}|p{3.15cm}|p{4.4cm}|}
\hline
& \multicolumn{1}{|c|}{\textsc{algebra}} & \multicolumn{1}{c|}{\textsc{dual objects}} & \multicolumn{1}{c|}{\textsc{reference}} \\ \hline
& $(\cl{O}_A; \zeta, \tau, \Delta, \nabla, \ast)$ & & \\
1 & -- subalgebras \newline with projections \newline (clones) & relations $R$ & Geiger~\cite{Geiger}; Bodnar\v{c}uk, \newline Kalu\v{z}\-nin, Kotov, Romov~\cite{BKKR} (finite domains), \newline Szabó~\cite{Szabo}; Pöschel~\cite{Poschel} \newline (general) \\[1ex]
2 & -- all subalgebras & relation pairs $(R,R')$ with $R' \subseteq R$ & Harnau~\cite{Harnau1-3} (finite domains) \newline {} \\[1ex]
3 & $(\cl{O}_A; \zeta, \tau, \Delta, \nabla)$ & constraints $(R,S)$ & Pippenger~\cite{Pippenger} (finite \newline domains), \newline Couceiro, Foldes~\cite{CF} (general) \\[1ex]
4 & $(\cl{O}_A; \zeta, \tau, \nabla)$ & generalized \newline constraints $(\phi, S)$ & Hellerstein~\cite{Hellerstein} (finite \newline domains), \newline Lehtonen~\cite{Lehtonen} (general) \\[1ex]
  & $(\cl{O}_A; \zeta, \tau, \nabla, \ast)$ & & \\
5 & -- subalgebras \newline with projections & clusters $\Phi$ & Lehtonen~\cite{Lehtonen} (general) \\
6 & -- all subalgebras & systems of pointed \newline multisets $(\Phi, \Phi')$ & Theorems~\ref{thm:pac}, \ref{thm:spms} \\
\hline
\end{tabular}
\end{center}
\bigskip
\caption{Galois theories for function algebras.}
\label{table:summary}
\end{table}
%%%%%%%%%%%%%%%%%%%%%%

Table~\ref{table:summary} summarizes the Galois connections that describe closure systems of subalgebras of various reducts of $(\cl{O}_A; \zeta, \tau, \Delta, \nabla, \ast)$ considered in the literature, up to our knowledge. It is well-known that the closure system of line~2 of Table~\ref{table:summary} (and, in particular, that of line~1) is countably infinite in the case $\card{A} = 2$ (see, e.g., \cite{Lau}). Pippenger~\cite{Pippenger} showed that the closure system of line~3 is uncountable whenever $\card{A} \geq 2$. By Theorem~\ref{thm:size}, this is also the case for the closure system of line~5. From this it follows that the closure systems of lines~4 and~6 are uncountable as well whenever $\card{A} \geq 2$.

Looking at possible directions for future work, we are inevitably drawn to consider the remaining reducts of $(\cl{O}_A; \zeta, \tau, \Delta, \nabla, \ast)$. This asks for analogous descriptions of the subalgebras of these reducts in terms of Galois connections and the sizes of the respective closure systems.

%%%%%%%%%%%%%%%%%%%%%%%%%%%%%%%%%%%%%%%%%%%%%%%
%% Acknowledgments (Optional)
%%%%%%%%%%%%%%%%%%%%%%%%%%%%%%%%%%%%%%%%%%%%%%%
\subsection*{Acknowledgements}

The authors wish to express their gratitude to Ivo Rosenberg for suggesting the problem of finding a characterization of the closure system of subalgebras of $(\cl{O}_A; \zeta, \tau, \nabla, \ast)$.

%%%%%%%%%%%%%%%%%%%%%%%%%%%%%%%%%%%%%%%%%%%%%%%%%%%%%%%%%%%%%%%%%%%
%%  BIBLIOGRAPHY                                                 %%
%%%%%%%%%%%%%%%%%%%%%%%%%%%%%%%%%%%%%%%%%%%%%%%%%%%%%%%%%%%%%%%%%%%

%%%%%%%%%%%%%%%%%%%%%%%%%%%%%%%%%%%%%%%%%%%%%%%%%%%%%%%%%%%%%%
\end{document}